\documentclass[12pt, a4paper]{amsart}
\usepackage{amsmath,amssymb}
\usepackage{enumerate}
\usepackage{mathrsfs}
 \usepackage{hyperref}
\newtheorem{theorem}{Theorem}[section]
\newtheorem{proposition}[theorem]{Proposition}
\newtheorem{corollary}[theorem]{Corollary}
\newtheorem{lemma}[theorem]{Lemma}
\theoremstyle{definition}

\newtheorem{definition}[theorem]{Definition}
\newtheorem{example}[theorem]{Example}
\newtheorem{question}[theorem]{Question}
 \usepackage{hyperref}

\numberwithin{equation}{section}

\DeclareMathOperator{\cl}{\textrm{Cl}}

\DeclareMathOperator{\inte}{\textrm{Int}}

\DeclareMathOperator{\card}{\textrm{Card}}
\DeclareMathOperator{\dom}{\textrm{dom}}

\begin{document}

\title[Algebraic structures in the family of non-Lebesgue measurable sets]{Algebraic structures in the family of non-Lebesgue measurable sets}
\author{Venuste NYAGAHAKWA}
\address{Department of Mathematics\\
University of RWANDA\\
PO. Box 3900, Kigali, Rwanda}
\email{venustino2005@yahoo.fr}

\author{Gratien HAGUMA}
\address{Department of Mathematics\\
University of RWANDA\\
PO. Box 3900, Kigali, Rwanda}
\email{hagugrat@yahoo.fr}

\author{Joseline MUNYANEZA}
\address{Department of General courses\\
RP-IPRC-HUYE\\
 9PRP+727, Butare}
\email{joselinemunyaneza@gmail.com}

\begin{abstract}
In the additive topological group $(\mathbb{R},+)$ of real numbers, we construct families of sets for which elements are not measurable in the Lebesgue sense. The constructed families have algebraic structures of being semigroups (i.e., closed under finite unions of sets), and invariant under the action of the group $\Phi(\mathbb{R})$ of all translations of $\mathbb{R}$ onto itself. Those semigroups are constructed by using Vitali selectors and Bernstein subsets on $\mathbb{R}$. In particular, we prove that the family $(\mathcal{S}(\mathcal{B})\vee \mathcal{S}(\mathcal{V}))*\mathcal{N}_0:=\{((U_1\cup U_2)\setminus N)\cup M: U_1\in \mathcal{S}(\mathcal{B}), U_2\in \mathcal{S}(\mathcal{V}), N,M\in \mathcal{N}_0\} $ is a semigroup of sets, invariant under the action of $\Phi(\mathbb{R})$ and consists of sets which are not measurable in the Lebesgue sense. Here, $\mathcal{S}(\mathcal{B})$ is the collection of all finite unions of some type of Bernstein subsets of $\mathbb{R}$, $\mathcal{S}(\mathcal{V})$ is the collection of all finite unions of Vitali selectors of $\mathbb{R}$, and $\mathcal{N}_0$ is the $\sigma$-ideal of all subsets of $\mathbb{R}$ having the Lebesgue measure zero.

\end{abstract}
\maketitle

\section{Introduction}
Let $(\mathbb{R},+)$ be the additive group of real numbers endowed with the Euclidean topology, and let $\mathcal{P}(\mathbb{R})$ be the collection of all subsets of $\mathbb{R}$. It is well-known that there exist subsets of $\mathbb{R}$ which are not measurable in the Lebesgue sense \cite{Kh1}, \cite{Ox};  for instance, Vitali selectors of $\mathbb{R}$, Bernstein sets of $\mathbb{R}$, as well as non-Lebesgue measurable subsets of $\mathbb{R}$ associated with Hamel basis.  Accordingly, the family $\mathcal{P}(\mathbb{R})$ can be decomposed into two disjoint non-empty families; namely, the family $\mathcal{L}(\mathbb{R})$ of all Lebesgue measurable subsets of $\mathbb{R}$, and the family $\mathcal{L}^c(\mathbb{R})=\mathcal{P}(\mathbb{R})\setminus \mathcal{L}(\mathbb{R})$ of all non-Lebesgue measurable subsets of $\mathbb{R}$. 

The algebraic structure; from the set-theoretic point of view, of the family $\mathcal{L}(\mathbb{R})$ is well-known. Indeed, the family $\mathcal{L}(\mathbb{R})$ is a $\sigma$-algebra of sets on $\mathbb{R}$, and hence it is closed under all basic set-operations. It contains the collection $\mathcal{B}_O(\mathbb{R})$ of all Borel subsets of $\mathbb{R}$, as well as, the collection $\mathcal{N}_0$ of all subsets of $\mathbb{R}$ having the Lebesgue measure zero. In addition, the family $\mathcal{L}(\mathbb{R})$ is invariant under the action of the group $\Phi (\mathbb{R})$ of all translations of $\mathbb{R}$ onto itself; it means that if $A\subseteq \mathbb{R}$ is such that $A\in \mathcal{L}(\mathbb{R})$ and $h\in \Phi(\mathbb{R})$ then $h(A)\in \mathcal{L}(\mathbb{R})$. 

On the other hand, the family $\mathcal{L}^c(\mathbb{R})$ does not have a well-defined structure from the set-theoretic point of view. In fact, the union (resp. intersection, difference, and symmetric difference) of two elements in the family $\mathcal{L}^c(\mathbb{R})$ can be inside or outside of $\mathcal{L}^c(\mathbb{R})$. However, like $\mathcal{L}(\mathbb{R})$, the family $\mathcal{L}^c(\mathbb{R})$ is invariant under the action of the group $\Phi (\mathbb{R})$.

Given a countable dense subgroup $Q$ of $(\mathbb{R},+)$, let $\mathcal{V}(Q)$ be the collection of all Vitali selectors related to $Q$. The construction of Vitali selectors is discussed in Subsection \ref{Eliyezeri} and more facts about them can be found in \cite{Kh1}. The following question constitutes the motivating key of this paper.

\begin{question}[\cite{Ch}]\label{Lebesgue11}
\emph{Could we find in $\mathcal{L}^c(\mathbb{R})$ subfamilies of $\mathcal{P}(\mathbb{R})$ containing $\mathcal{V}(Q)$ and have some algebraic structures from the set-theoretic point of view?}
\end{question}

In \cite{CN2}, it was shown that each element of the family $\mathcal{V}_1(Q)=\{\bigcup_{i=1}^n V_i: V_i\in \mathcal{V}(Q), n\in \mathbb{N}\}$ of all finite unions of Vitali selectors related to $Q$, is a semigroup of sets with respect to the operation of union of sets, and that, it is invariant under the action of the group $\Phi(\mathbb{R})$ such that $\mathcal{V}(Q)\subsetneq \mathcal{V}_1(Q)\subsetneq \mathcal{L}^c(\mathbb{R})$. In addition, the family $\mathcal{V}_2(Q)=\mathcal{V}_1(Q)*\mathcal{N}_0:=\{(U\setminus M)\cup N: U\in \mathcal{V}_1(Q), M, N\in \mathcal{N}_0\}$ was shown to be a semigroup of sets, invariant under the action of the group $\Phi(\mathbb{R})$ and that  $\mathcal{V}_1(Q)\subsetneq \mathcal{V}_2(Q)\subsetneq \mathcal{L}^c(\mathbb{R})$.  Using the family $\mathcal{V}_1(Q)$ and different ideals of subsets of $\mathbb{R}$, different semigroups of sets for which elements are not measurable in the Lebesgue sense, were constructed in \cite{CN4} and \cite{CN1}. 

Let $\mathcal{V}$ be the family of all Vitali selectors of $\mathbb{R}$, and let  $\mathcal{S}(\mathcal{V})$ be the collection of all finite unions of elements of $\mathcal{V}$; that is,  $\mathcal{S}(\mathcal{V})=\{ \bigcup_{i=1}^n V_i:  V_i\in \mathcal{V}, n\in \mathbb{N}\}$. In \cite{VH}, it is shown that  the family
 $\mathcal{S}(\mathcal{V})*\mathcal{N}_0:=\{(U\setminus M)\cup N: U\in \mathcal{S}(\mathcal{V}), M, N\in \mathcal{N}_0\}$ is an abelian semigroup of sets, for which elements are not measurable in the Lebesgue sense, and that $\mathcal{V}_2(Q)\subsetneq \mathcal{S}(\mathcal{V})*\mathcal{N}_0$ for every countable dense subgroup $Q$ of $(\mathbb{R},+)$.  We note that the non-Lebesgue measurability of elements of the family $\mathcal{S}(\mathcal{V})*\mathcal{N}_0$ was also proved in \cite{CT}  together with other interesting facts about the semigroups generated by Vitali selectors of $\mathbb{R}$. Besides that, it is proved in \cite{VH} that the abelian semigroup $\mathcal{S}(\mathcal{V})*\mathcal{N}_0$ has an algebraic structure of being invariant under the action of the group $\Pi(\mathbb{R})$ of all affine transformations of $\mathbb{R}$ onto itself. 
 
The semigroups of sets that exist in the literature (for which elements are not measurable in the Lebesgue sense), are mostly constructed by using  Vitali selectors of $\mathbb{R}$. In this paper, we consider a more general setting, by looking away for extending Question \ref{Lebesgue11}. Accordingly,  we consider a Bernstein set $B$ which has an algebraic structure of being a subgroup of $(\mathbb{R},+)$. Such a set exists as it is shown in \cite{Kh1} and \cite{MC}. Furthermore,  we consider the collection $\mathbb{R}/B$ of all  cosets (translates) of $B$ in $(\mathbb{R},+)$, that we denote by $\mathcal{B}$ for simplicity. Since the family of Bernstein sets is preserved by homeomorphisms of $\mathbb{R}$ onto itself \cite{Kh1}, it follows that each element of $\mathcal{B}$ is also a  Bernstein set.  

\begin{question}\label{Lebesgue12}
\emph{Could we find in $\mathcal{L}^c(\mathbb{R})$ subfamilies of $\mathcal{P}(\mathbb{R})$ containing $\mathcal{B}$ and have some algebraic structures from the set-theoretic point of view?}
\end{question}

Through the paper, different families of sets answering Question \ref{Lebesgue12} are constructed.  In particular, the family $\mathcal{S}(\mathcal{B})=\{\bigcup_{i=1}^n B_i: B_i\in \mathcal{B}, n\in \mathbb{N}\}$ of all finite unions of elements of $\mathcal{B}$, and its extension $\mathcal{S}(\mathcal{B})*\mathcal{N}_0$ by the $\sigma$-ideal $\mathcal{N}_0$, constitute an  answer to Question \ref{Lebesgue12}. 

For an extension, one can consider  the family  $\mathcal{V}\vee \mathcal{B}:=\{V\cup B: V\in \mathcal{V} \text{ and } B\in \mathcal{B}\}$ made by all possible pairwise unions of a Vitali selector and a Bernstein set. The main aim of this paper, is to  construct families of sets that constitute answers to the following question; which generalizes in some sense Question \ref{Lebesgue11} and Question \ref{Lebesgue12}.

\begin{question}\label{Lebesgue13}
\emph{Could we find in $\mathcal{L}^{c}(\mathbb{R})$ subfamilies of $\mathcal{P}(\mathbb{R})$ containing  the family $\mathcal{V}\vee \mathcal{B}$ and have some algebraic structures from the set-theoretic point of view?}
\end{question}

In this paper, different families of sets having the algebraic structure of being semigroups with respect to the operation of the union of sets, are constructed through the use of $\mathcal{V}$ and $\mathcal{B}$. The constructed semigroups extend the existing ones, and in particular, some contain the family $\mathcal{V}\vee \mathcal{B}$. The constructed families also consist of sets that are not measurable in the Lebesgue sense, and they are invariant under the action of the group $\Phi(\mathbb{R})$. It is proved that, for any Bernstein sets $B_1$ and $B_2$ having the algebraic structures of being subgroups of $(\mathbb{R},+)$, the family $[\mathcal{S}(\mathcal{B}_1)\vee \mathcal{S}(\mathcal{B}_2)\vee \mathcal{S}(\mathcal{V})]*\mathcal{N}_0:=\{[(U_1\cup U_2\cup U_3)\setminus N]\cup M: U_1\in \mathcal{S}(\mathcal{B}_1), U_2\in \mathcal{S}(\mathcal{B}_2), U_3\in \mathcal{S}(\mathcal{V}), N,M\in \mathcal{N}_0\}$ is an abelian semigroup of sets, for which elements are not measurable in the Lebesgue sense, and that it is invariant under the action of the group $\Phi(\mathbb{R})$. We further show that the equality $[\mathcal{S}(\mathcal{B}_1)\vee \mathcal{S}(\mathcal{B}_2)\vee \mathcal{S}(\mathcal{V})]*\mathcal{N}_0=(\mathcal{S}(\mathcal{B}_1)*\mathcal{N}_0)\vee (\mathcal{S}(\mathcal{B}_2)*\mathcal{N}_0)\vee (\mathcal{S}(\mathcal{V})*\mathcal{N}_0)$ always holds.

This paper uses the standard notation and facts from Set Theory and Real Analysis, and it is organized as follows: After an introductory section, where motivating ideas and the problem under investigation are developed, the second section deals with the theory of semigroups, the theory of Vitali selectors, and different facts about Bernstein sets. The third section is about the semigroups of non-Lebesgue measurable sets constructed by using Bernstein sets. The fourth section complements the part about Vitali selectors, developed in the preliminary section, and it is about semigroups of non-Lebesgue measurable sets generated by Vitali selectors. The fifth section, is devoted to the semigroups of non-Lebesgue measurable sets, that are constructed by using Vitali selectors and Bernstein sets, simultaneously.

\section{Preliminary facts}
\subsection{Theory of semigroups and ideals of sets}
Let $\mathcal{S}$ be a non-empty set.  The set $\mathcal{S}$  is called a \emph{semigroup} of sets, if there is a binary operation  $\ast: \mathcal{S}\times \mathcal{S}\longrightarrow \mathcal{S}$, for which the associativity law is satisfied; i.e., $(x\ast y)\ast z =x\ast (y\ast z)$ for all $x,y,z\in \mathcal{S}$. The semigroup $\mathcal{S}$ is said to be \emph{abelian}, if $x\ast y=y\ast x$ for all $x,y\in \mathcal{S}$.

For a non-empty set $X$, let $\mathcal{P}(X)$ be the collection of all subsets of $X$. Consider a non-empty family of sets $\mathcal{S}\subseteq \mathcal{P}(X)$, such that for each pair of elements $A,B\in \mathcal{S}$, we have $A\cup B\in \mathcal{S}$. Since the union of sets is both commutative and associative, such a family of sets, is an abelian semigroup, with respect to the operation of union of sets. 

\begin{definition}\label{Lebesgue21}
A non-empty collection of sets $\mathcal{S}\subseteq \mathcal{P}(X)$, is called a \emph{semigroup of sets} on $X$, if it is closed under finite unions. If  $\mathcal{S}$ is closed under countable unions,  then it is said to be a \emph{$\sigma$-semigroup} of sets on $X$.
\end{definition}

It is evident that if $\mathcal{S}$ is a semigroup of sets on $X$, with respect to the operation of union of sets, then the collection $\{X\setminus S: S\in \mathcal{S}\}=\{S^c: S\in \mathcal{S}\}$ of all complements of elements of $\mathcal{S}$ in $X$, is closed under finite intersection of sets, and thus, it is a  semigroup of sets with respect to the set-theoretic operation of intersection of sets on $X$. 

A non-empty family of sets  $\mathcal{R}$ is called a \emph{ring} of sets on $X$, if it is a semigroup of sets and closed under relative complement; that is if $A, B\in \mathcal{R}$ then $A\cup B\in \mathcal{R}$ and if $A,B\in \mathcal{R}$ then $A\setminus B\in \mathcal{R}$.

Recall \cite{Mo} that a non-empty collection $\mathcal{I}\subseteq \mathcal{P}(X)$ of sets  is called an \emph{ideal of sets} on $X$, if it satisfies the following conditions:
\begin{enumerate}[(i)]
\item If $A\in \mathcal{I}$ and $B\in \mathcal{I}$, then $A\cup B\in \mathcal{I}$.
\item If $A\in \mathcal{I}$ and $B\subseteq A$, then $B\in \mathcal{I}$.
\end{enumerate} 
If the ideal of sets $\mathcal{I}$ is closed under countable unions of sets, then it is called a \emph{$\sigma$-ideal of sets} on $X$. Clearly, each ideal of sets is a semigroup of sets which is closed under taking subsets.

\begin{example}[\cite{CN4}]\label{Lebesgue22}
If $\mathcal{A}\subseteq \mathcal{P}(X)$ is a non-empty family of sets,  consider the collection $\mathcal{S}({\mathcal{A}})=\left\{\bigcup_{i=1}^nA_i: A_i\in \mathcal{A}, n\in \mathbb{N}\right\}$ of all finite unions of elements of $\mathcal{A}$, and
consider the collection $\mathcal{I}({\mathcal{A}})=\{B\in \mathcal{P}(X): \text{there is  }A\in \mathcal{S}({\mathcal{A}})  \text{  such that  }B\subseteq A\}$. It is clear that the family $\mathcal{S}({\mathcal{A}})$ is a semigroup of sets, while the collection $\mathcal{I}({\mathcal{A}})$ is an ideal of sets on $X$. The family $\mathcal{S}({\mathcal{A}})$ is called the \emph{semigroup of sets generated by $\mathcal{A}$}, while $\mathcal{I}({\mathcal{A}})$ is called the  \emph{ideal of sets generated by $\mathcal{A}$}. Evidently the inclusions $\mathcal{S}(\mathcal{A})\subseteq \mathcal{I}({\mathcal{A}})$ and  $\mathcal{A}\subseteq \mathcal{S}(\mathcal{A})$ always hold. If $\mathcal{A}$ is a semigroup of sets on $X$, then we have the equality $\mathcal{S}(\mathcal{A})=\mathcal{A}$.
\end{example}

\begin{example}\label{Lebesgue23}
On the set $\mathbb{R}$ of real numbers, consider the following families of sets: The family $\mathcal{I}_f$ of all finite subsets of $\mathbb{R}$, the family $\mathcal{I}_c$ of all countable subsets of $\mathbb{R}$, and the family $\mathcal{B}_b(\mathbb{R})$ of all bounded subsets of $\mathbb{R}$. All these families are semigroups of sets on $\mathbb{R}$. They are ideals of sets on $\mathbb{R}$ and both are rings of sets on $\mathbb{R}$.
\end{example}

For any families  $\mathcal{A}$ and $\mathcal{B}$ of subsets of $X$, define a new family of sets on $X$ by setting $\mathcal{A}* \mathcal{B}=\{(A\setminus B_1)\cup B_2: A\in \mathcal{A}, B_1\in \mathcal{B}, B_2\in \mathcal{B}\}$. One can observe that, if $\mathcal{A}, \mathcal{B}, \mathcal{C}$ and $\mathcal{D}$ are families of sets on $X$ such that $\mathcal{A}\subseteq \mathcal{B}, \mathcal{C}\subseteq \mathcal{D}$, then $\mathcal{A}* \mathcal{C} \subseteq \mathcal{B}* \mathcal{D}$. For any family of sets $\mathcal{A}$ on $X$, the inclusion $\mathcal{A}\subseteq \mathcal{A}* \mathcal{A}$ always holds, but  the reverse inclusion does not need to hold. If $\mathcal{S}_1$ and $\mathcal{S}_2$ are semigroups of sets, then the family $\mathcal{S}_1*\mathcal{S}_2$ does not need to be a semigroup of sets, unless one of them is an ideal of sets. This fact is illustrated by the following statement, which presents a way of extending a given semigroup by using an ideal of sets.

\begin{proposition}[\cite{CN4}]\label{Lebesgue25}
Let $\mathcal{S}$ be a semigroup of sets on $X$ and let $\mathcal{I}$ be an ideal of sets on $X$. Then, the families $\mathcal{I}* \mathcal{S}$  and $\mathcal{S}* \mathcal{I}$ are semigroups of sets on $X$, such that $\mathcal{S}\subseteq \mathcal{I}* \mathcal{S}\subseteq \mathcal{S}* \mathcal{I}$. Moreover,  $\mathcal{I}* (\mathcal{I}* \mathcal{S})=\mathcal{I}* \mathcal{S}$ and $(\mathcal{S}* \mathcal{I})* \mathcal{I}=\mathcal{S}* \mathcal{I}$. 
\end{proposition}

For any families  $\mathcal{A}$ and $\mathcal{B}$ of subsets of $X$, consider the collection $\mathcal{A}\vee \mathcal{B}=\{A\cup B: A\in \mathcal{A}, B\in \mathcal{B}\}$. It is clear that if $\mathcal{A}, \mathcal{B}, \mathcal{C}$ and $\mathcal{D}$ are families  of sets on $X$, such that $\mathcal{A}\subseteq \mathcal{B}, \mathcal{C}\subseteq \mathcal{D}$, then $\mathcal{A}\vee \mathcal{C} \subseteq \mathcal{B}\vee \mathcal{D}$. For any family of sets $\mathcal{A}$ on $X$, the inclusion $\mathcal{A}\subseteq \mathcal{A}\vee \mathcal{A}$ always hold but the reverse inclusion doesn't need to hold. If $\mathcal{A}$ is a semigroup of sets on $X$, then $\mathcal{A}\vee \mathcal{A}=\mathcal{A}$.  The inclusions $\mathcal{A}\subseteq \mathcal{A}\vee \mathcal{B}$ and $\mathcal{B}\subseteq \mathcal{A}\vee \mathcal{B}$ do not need to hold for any families $\mathcal{A}$ and $\mathcal{B}$ with or without the assumption of being semigroups. If $\mathcal{S}_1$ and $\mathcal{S}_2$ are semigroups of sets, then the usual union $\mathcal{S}_1\cup \mathcal{S}_2$ does not need to be a semigroup of sets, however the following lemma holds.

\begin{lemma}\label{Lebesgue26}
If  $\mathcal{S}_1$ and $\mathcal{S}_2$ are semigroups  of sets on $X$, then the family $\mathcal{S}_1\vee \mathcal{S}_2$ is also a semigroup of sets on $X$. 
\end{lemma}
\begin{proof}
Assume that $U_1\in \mathcal{S}_1\vee \mathcal{S}_2$ and $U_2\in \mathcal{S}_1\vee \mathcal{S}_2$.  Then $U_1=A_1\cup A_2$ and $U_2=B_1\cup B_2$ for some $A_1, B_1\in \mathcal{S}_1$ and $A_2, B_2\in \mathcal{S}_2$. Note that  $A_1\cup B_1\in \mathcal{S}_1$ and $A_2\cup B_2\in \mathcal{S}_2$ as $\mathcal{S}_1$ and $\mathcal{S}_2$ are semigroups of sets.  It follows that $U_1\cup U_2=(A_1\cup A_2)\cup (B_1\cup B_2)=(A_1\cup B_1)\cup (A_2\cup B_2)$ is an element of $\mathcal{S}_1\vee \mathcal{S}_2$. 
\end{proof}

We observe that for any sets $A$ and $B$, we have $A\cup B=(A\setminus B)\cup B=(B\setminus A)\cup A$, which implies that $\mathcal{A}\vee \mathcal{B}\subseteq \mathcal{A}*\mathcal{B}$ and $\mathcal{A}\vee \mathcal{B}\subseteq \mathcal{B}*\mathcal{A}$, for any families of sets $\mathcal{A}$ and $\mathcal{B}$ on $X$. In addition, the equality $\mathcal{S}(\mathcal{A})\cup \mathcal{S}(\mathcal{B})=\mathcal{S}(\mathcal{A} \cup \mathcal{B})$ does not need to hold.

\begin{example}\label{Lebesgue27}
On the set  $X=\{a,b,c,d\}$, let $A=\{a,b\}, B=\{b,c\}$ and $D=\{c,d\}$, and consider the families $\mathcal{A}=\{A\}$ and $\mathcal{B}=\{B,D\}$. Note that $\mathcal{A}\cup \mathcal{B}=\{A,B,D\}$, $\mathcal{S}(\mathcal{A})=\mathcal{A}$ and $\mathcal{S}(\mathcal{B})=\{B, D, B\cup D\}$. It is clear that $\mathcal{S}(\mathcal{A})\cup \mathcal{S}(\mathcal{B})=\{A,B,D,B\cup D\}$ while $\mathcal{S}(\mathcal{A}\cup \mathcal{B})=\{A,B,D,A\cup B, A\cup D, B\cup D, A\cup B\cup D\}=\{A,B,D,A\cup B, B\cup D, X\}$. Hence $\mathcal{S}(\mathcal{A})\cup \mathcal{S}(\mathcal{B})\subsetneq \mathcal{S}(\mathcal{A} \cup \mathcal{B})$. Since $\mathcal{A}\vee \mathcal{B}=\{A\cup B, A\cup D\}=\{A\cup B, X\}$, we further remark that  $\mathcal{S}(\mathcal{A})\vee \mathcal{S}(\mathcal{B})=\{A\cup B, A\cup D, A\cup B\cup D\}=\{A\cup B, X\}=\mathcal{S}(\mathcal{A}\vee \mathcal{B})$.
\end{example}

\begin{lemma}\label{Lebesgue28}
If $\mathcal{A}$ and $\mathcal{B}$ are non-empty families of sets on $X$, then the equality $\mathcal{S}(\mathcal{A \vee \mathcal{B}})= \mathcal{S}(\mathcal{A}) \vee \mathcal{S}(\mathcal{B})$ always holds. 
\end{lemma}

\begin{proof}
Assume that  $Y\in \mathcal{S}(\mathcal{A}\vee \mathcal{B})$. Then $Y=\bigcup_{i=1}^n Y_i$, where $Y_i\in \mathcal{A}\vee \mathcal{B}$, i.e. $Y_i=A_i\cup B_i$ with $A_i\in \mathcal{A}$, $B_i\in \mathcal{B}$ and $n\in \mathbb{N}$. Hence $Y=\bigcup_{i=1}^n \left(A_i\cup B_i\right)=\left( \bigcup_{i=1}^n A_i\right) \cup \left( \bigcup_{i=1}^n B_i\right)$. Put $A= \bigcup_{i=1}^n A_i$ and $B=\bigcup_{i=1}^n B_i$. It is clear that $A\in \mathcal{S}(\mathcal{A})$ and $B\in \mathcal{S}(\mathcal{B})$, and hence  $Y\in \mathcal{S}(\mathcal{A})\vee \mathcal{S}(\mathcal{B})$ implying that  $\mathcal{S}(\mathcal{A \vee \mathcal{B}})\subseteq  \mathcal{S}(\mathcal{A}) \vee \mathcal{S}(\mathcal{B})$.

Assume that $Y\in  \mathcal{S}(\mathcal{A})\vee \mathcal{S}(\mathcal{B})$. Then $Y=A\cup B$, where $A\in \mathcal{S}(\mathcal{A})$ and $B\in \mathcal{S}(\mathcal{B})$, which means that $A=\bigcup_{i=1}^n A_i$ and $B=\bigcup_{i=1}^m B_i$, where $A_i\in \mathcal{A}$ and $B_i\in \mathcal{B}$ for some $n$ and $m$ in $\mathbb{N}$. We will have two cases to consider:
\begin{enumerate}[$\bullet$]
\item If $n=m$, then $Y=\bigcup_{i=1}^n\left(A_i\cup B_i\right)$ and hence $Y\in \mathcal{S}(\mathcal{A}\vee \mathcal{B})$. 
\item If $n\neq m$, without loosing generality, assume that $n<m$. We can write
$Y=\left[\bigcup_{i=1}^n \left(A_i\cup B_i\right)\right] \cup \left( \bigcup_{i=n+1}^m B_i\right)$. For $i=n+1,n+2,\cdots, m$, put  $A_i=A_k$, where $k$ is some fixed integer in the set $\{1,2,\cdots, n\}$. It follows that $Y=\left[\bigcup_{i=1}^n \left(A_i\cup B_i\right)\right]\cup \left[\bigcup_{i=n+1}^m(A_i\cup B_i)\right]=\bigcup_{i=1}^m \left(A_i\cup B_i\right)$. Since $A_i\cup B_i\in \mathcal{A}\vee \mathcal{B}$ for $i=1,2, \cdots,m$, it follows that  $Y\in \mathcal{S}(\mathcal{A}\vee \mathcal{B})$, and thus $\mathcal{S}(\mathcal{A}) \vee \mathcal{S}(\mathcal{B})\subseteq  \mathcal{S}(\mathcal{A \vee \mathcal{B}})$.
\end{enumerate}
\end{proof}
The following proposition is a generalization of Lemma \ref{Lebesgue28}  for any finite collection of families of sets.

\begin{proposition}\label{Lebesgue29}
Let  $\mathcal{A}_i$ be  a non-empty family of sets on $X$, where $i=1,2, \cdots, n$,  for some $n\in \mathbb{N}$. Then, the equality  $\mathcal{S}\left( \bigvee_{i=1}^n \mathcal{A}_i\right)= \bigvee_{i=1}^n \mathcal{S}(\mathcal{A}_i)$ always holds.
\end{proposition}

It follows from Proposition \ref{Lebesgue25}, that if $\mathcal{S}_1$ and $\mathcal{S}_2$ are semigroups of sets, then the families $(\mathcal{S}_1\vee \mathcal{S}_2)*\mathcal{I}$  and $\mathcal{I}*(\mathcal{S}_1\vee \mathcal{S}_2)$  are semigroups of sets for any ideal of sets $\mathcal{I}$ on $X$. However, no one of the inclusions
 $\mathcal{S}_1*\mathcal{I}\subseteq (\mathcal{S}_1\vee \mathcal{S}_2)*\mathcal{I}$, $\mathcal{S}_2*\mathcal{I}\subseteq (\mathcal{S}_1\vee \mathcal{S}_2)*\mathcal{I}$, $\mathcal{I}*\mathcal{S}_1\subseteq \mathcal{I}*(\mathcal{S}_1\vee \mathcal{S}_2)$ and $\mathcal{I}*\mathcal{S}_2\subseteq \mathcal{I}*(\mathcal{S}_1\vee \mathcal{S}_2)$ needs to hold.
 
\begin{example}\label{Lebesgue210}
Let $X$ be a non-empty set having atleast three elements, i.e. $\card (X)\geq 3$, and let $A$ be a non-empty proper subset of $X$. Let $B=X\setminus A$. Consider  the semigroups $\mathcal{S}_1=\{A, X\}$, $\mathcal{S}_2=\{B, X\}$ and the ideals of sets $\mathcal{I}=\mathcal{P}(A)$ and $\mathcal{K}=\mathcal{P}(B)$ on $X$.  It is clear that $\mathcal{S}_1\vee \mathcal{S}_2=\{X\}$, and $\emptyset$ and $A$ cannot be elements of $ (\mathcal{S}_1\vee \mathcal{S}_2)*\mathcal{I}$, but $\emptyset$ and $A$ are elements of $\mathcal{S}_1*\mathcal{I}$.  Similarly, the collection $\mathcal{S}_2*\mathcal{K}$ contains the elements $\emptyset$ and $B$, but the family $ (\mathcal{S}_1\vee \mathcal{S}_2)*\mathcal{K}$ cannot contain $\emptyset$ and $B$. Hence $\mathcal{S}_1*\mathcal{I}\nsubseteq (\mathcal{S}_1\vee \mathcal{S}_2)*\mathcal{I}$, and $\mathcal{S}_2*\mathcal{K}\nsubseteq (\mathcal{S}_1\vee \mathcal{S}_2)*\mathcal{K}$.   Further, observe that $\mathcal{I}*(\mathcal{S}_1\vee \mathcal{S}_2)=\{X\}=\mathcal{K}*(\mathcal{S}_1\vee \mathcal{S}_2)$. The semigroup  $\mathcal{I}*\mathcal{S}_2$ contains the set $B$ and  the semigroup  $\mathcal{K}*\mathcal{S}_1$ contains the set $A$. Hence $\mathcal{I}*\mathcal{S}_2\nsubseteq \mathcal{I}*(\mathcal{S}_1\vee \mathcal{S}_2)$ and  $\mathcal{K}*\mathcal{S}_1\nsubseteq \mathcal{K}*(\mathcal{S}_1\vee \mathcal{S}_2)$.
\end{example}

\begin{proposition}\label{Lebesgue211}
Let $\mathcal{S}_1$ and $\mathcal{S}_2$ be semigroups of sets on $X$. If $\mathcal{I}$ is an ideal of sets on $X$, then the following equalities always hold:

\begin{enumerate}[(i)]
\item $(\mathcal{S}_1\vee \mathcal{S}_2)*\mathcal{I}=(\mathcal{S}_1*\mathcal{I})\vee (\mathcal{S}_2*\mathcal{I})$.
\item $ \mathcal{I}*(\mathcal{S}_1\vee \mathcal{S}_2)=(\mathcal{I}*\mathcal{S}_1)\vee (\mathcal{I}*\mathcal{S}_2)$.
\end{enumerate}
\end{proposition}

\begin{proof}
\begin{enumerate}[(i).]
\item Assume that $A\in (\mathcal{S}_1\vee \mathcal{S}_2)*\mathcal{I}$. Then $A=[(S_1\cup S_2)\setminus I]\cup K $ where $S_1\in \mathcal{S}_1, S_2\in \mathcal{S}_2$ and $I, K\in \mathcal{I}$. It is clear that $A=(S_1\setminus I)\cup (S_2\setminus I)\cup K=[(S_1\setminus I)\cup K]\cup [(S_2\setminus I)\cup K]\in (\mathcal{S}_1*\mathcal{I})\vee (\mathcal{S}_2*\mathcal{I})$.

Assume that $A\in (\mathcal{S}_1*\mathcal{I})\vee (\mathcal{S}_2*\mathcal{I})$. Then $A=[(S_1\setminus N)\cup L]\cup [(S_2\setminus P)\cup R]$, where $S_1\in \mathcal{S}_1, S_2\in \mathcal{S}_2$ and $N,L,P,R \in \mathcal{I}$. Note that
 $A=(S_1\setminus N)\cup (S_2\setminus P)\cup (L\cup R)$. Putting $I=L\cup R\in \mathcal{I}$ it follows that  $A=\left[(S_1\cap N^{c})\cup (S_2\cap P^{c})\right]^{cc}\cup I=   \left[(S_1\cap N^{c})^{c}\cap (S_2\cap P^{c})^{c}\right]^{c}\cup I= \left[(S_1^{c}\cup N)\cap (S_2^{c}\cup P)\right]^{c}\cup I$. Furthermore,
$A=[(S_1^{c}\cap S_2^{c})\cup \left((S_1^{c}\cap P)\cup (S_2^{c}\cap N)\cup (N\cap P)\right)]^{c}\cup I $. Put $J=(S_1^{c}\cap P)\cup (S_2^{c}\cap N)\cup (N\cap P)\in \mathcal{I} $ and note that  $A=\left[ (S_1^c\cap S_2^c)\cup J   \right]^c\cup I=\left[(S_1^c\cap S_2^c)^c\cap J^c\right]\cup I=[(S_1\cup S_2)\setminus J]\cup I$. Since $S_1\in \mathcal{S}_1, S_2 \in \mathcal{S}_2 $ and $J,I \in \mathcal{I}$ then we have $A\in (\mathcal{S}_1\vee \mathcal{S}_2)*\mathcal{I}$.

\item Assume that  $A\in (\mathcal{I}*\mathcal{S}_1)\vee (\mathcal{I}*\mathcal{S}_2)$. Then $A=\left[(I_1\setminus U_1)\cup W_1\right] \cup \left[ (I_2\setminus U_2)\cup W_2\right]$ where $I_1, I_2\in \mathcal{I}, U_1, W_1\in \mathcal{S}_1$ and $U_2, W_2\in \mathcal{S}_2$. It is clear that $A=(I_1\setminus U_1)\cup (I_2\setminus U_2)\cup (W_1\cup W_2)=I\cup (W_1\cup W_2)$, where $I=(I_1\setminus U_1)\cup (I_2\setminus U_2)$. Since  $I\in \mathcal{I}$ then the set $A$ can be written as $A=\left [I\setminus (W_1\cup W_2)\right]\cup (W_1\cup W_2) \in \mathcal{I}*(\mathcal{S}_1\vee \mathcal{S}_2)$.

Assume that $A\in \mathcal{I}*\left(\mathcal{S}_1\vee \mathcal{S}_2\right)$. Then  $A=\left[I\setminus (U_1\cup U_2)\right] \cup (W_1\cup W_2)$ where $I\in \mathcal{I}, U_1, W_1\in \mathcal{S}_1$ and $U_2, W_2\in \mathcal{S}_2$. Since $I\setminus (U_1\cup U_2)=((I\setminus U_2)\setminus U_1))\cup ((I\setminus U_1)\setminus U_2)$, then we get $A=\left[(I\setminus U_2)\setminus U_1\right]\cup \left[(I\setminus U_1)\setminus U_2\right]\cup (W_1\cup W_2)=\left[\left((I\setminus U_2)\setminus U_1\right) \cup W_1\right]\cup \left[\left((I\setminus U_1)\setminus U_2\right)\cup W_2\right]$. Since  the sets $I\setminus U_1$ and $I\setminus U_2$ are elements of $\mathcal{I}$ then we have $A\in (\mathcal{I}*\mathcal{S}_1)\vee (\mathcal{I}*\mathcal{S}_2)$.
\end{enumerate}
\end{proof}

\begin{corollary}\label{Lebesgue212}
Let  $\mathcal{S}_1, \mathcal{S}_2, \cdots, \mathcal{S}_n$ be a finite collection of semigroups of sets on $X$. If $\mathcal{I}$ is an ideal of sets on $X$, then the following equalities hold:
\begin{enumerate}[(i).]
\item $\left(\bigvee_{i=1}^n \mathcal{S}_i\right)*\mathcal{I}=\bigvee_{i=1}^{n} \left( \mathcal{S}_i*\mathcal{I} \right)$.
\item $\mathcal{I}*\left(\bigvee_{i=1}^n \mathcal{S}_i\right)=\bigvee_{i=1}^{n} \left( \mathcal{I}*\mathcal{S}_i \right)$.
\end{enumerate}
\end{corollary}

It follows from Proposition \ref{Lebesgue211} that if  $\mathcal{A}$ and $\mathcal{B}$ are semigroups  of sets on $X$, and $\mathcal{I}$ is an ideal of sets on $X$, then  $\mathcal{S}[(\mathcal{A}\vee \mathcal{B})*\mathcal{I}]=\mathcal{S}\left[(\mathcal{A}*\mathcal{I})\vee (\mathcal{B}*\mathcal{I})\right]$ and $\mathcal{S}[\mathcal{I}*(\mathcal{A}\vee \mathcal{B})]=\mathcal{S}\left[(\mathcal{I}*\mathcal{A})\vee (\mathcal{I}*\mathcal{B})\right]$.

\begin{question}\label{Lebesgue213}
\emph{Let $\mathcal{A}$ and $\mathcal{B}$ be families of sets, and let $\mathcal{I}$ be an ideal of sets. What is the relationship (in the sense of inclusion) between the semigroups of sets $\mathcal{S}(\mathcal{A}\vee \mathcal{B})*\mathcal{I}$ and $\mathcal{S}[(\mathcal{A}\vee \mathcal{B})*\mathcal{I}]$?}
\end{question}

We observe that if $\mathcal{A}$ and $\mathcal{B}$ are semigroups of sets then by Example \ref{Lebesgue22},  we get $\mathcal{S}(\mathcal{A}\vee \mathcal{B})*\mathcal{I}= (\mathcal{A}\vee \mathcal{B})*\mathcal{I}=\left(\mathcal{A}*\mathcal{I}\right)\vee \left(\mathcal{B}*\mathcal{I}\right)=\mathcal{S}[(\mathcal{A}\vee \mathcal{B})*\mathcal{I}]$. 

The following statement can be used in the extension of  semigroup of sets. Its proof is based on the properties of ideals of sets goes in the same line as Proposition \ref{Lebesgue211}.

\begin{proposition}\label{Lebesgue214}
Let $\mathcal{I}_{1}$ and $\mathcal{I}_{2}$ be ideals of sets on $X$ and let $\mathcal{S}$ be a semigroup of sets on $X$. Then, the following hold:
	\begin{enumerate}[(i)]
		\item $\mathcal{S}* \mathcal{I}_{i}\subseteq \mathcal{S}* (\mathcal{I}_{1}\vee \mathcal{I}_{2})$ for $i=1,2$ and 
$(\mathcal{S}*\mathcal{I}_{1})\vee (\mathcal{S}*\mathcal{I}_{2}) \subseteq \mathcal{S}* (\mathcal{I}_{1}\vee \mathcal{I}_{2})$.
		\item $\mathcal{I}_{i}* \mathcal{S}\subseteq (\mathcal{I}_{1}\vee \mathcal{I}_{2})* \mathcal{S}= (\mathcal{I}_{1}*\mathcal{S})\vee (\mathcal{I}_{2}* \mathcal{S}) \text{ for }i=1,2$.
	\end{enumerate}
\end{proposition}
\begin{proof} 
\begin{enumerate}[(i).]
\item Since $\mathcal{I}_{1} \subseteq \mathcal{I}_{1}\vee \mathcal{I}_{2}$ and  $\mathcal{I}_{2} \subseteq \mathcal{I}_{1}\vee \mathcal{I}_{2}$, then inclusions $\mathcal{S}* \mathcal{I}_{1}\subseteq \mathcal{S}* (\mathcal{I}_{1}\vee \mathcal{I}_{2})$ and $\mathcal{S}* \mathcal{I}_{2}\subseteq \mathcal{S}* (\mathcal{I}_{1}\vee \mathcal{I}_{2})$ follow directly.

Assume that $A\in (\mathcal{S}*\mathcal{I}_{1})\vee (\mathcal{S}*\mathcal{I}_{2})$. Then  $A= [(S_{1}\setminus I_{1})\cup I_{2}] \cup [(S_{2}\setminus I_{3})\cup I_{4}]$, where $S_{1}, S_{2} \in \mathcal{S}, I_{1}, I_{2}\in \mathcal{I}_{1}$ and $I_{3}, I_{4}\in \mathcal{I}_{2}$.  We can write that $A= (S_{1}\setminus I_{1}) \cup (S_{2}\setminus I_{3}) \cup (I_{2}\cup I_{4})$.  It is clear that  $(S_{1}\setminus I_{1}) \cup (S_{2}\setminus I_{3})= \left[ (S_{1}\setminus I_{1})\cup (S_{2}\setminus I_{3})\right]^{cc} = \left[ \left[(S_{1}\setminus I_{1})\cup (S_{2}\setminus I_{3}) \right]^{c}\right]^{c} = \left[ (S_{1}\cap I_{1}^{c})^{c}\cap (S_{2}\cap I_{3}^{c})^{c}\right]^{c}$, and we can easily write the following expression: $\left[(S_{1}^{c}\cup I_{1})\cap (S_{2}^{c}\cup I_{3})\right]^{c}=$\\
$\left[ (S_{1}^{c}\cap S_{2}^{c}) \cup (S_{1}^{c}\cap I_{3})\cup (S_{2}^{c}\cap I_{1})\cup (I_{1}\cap I_{3})\right]^{c}$. Since $ S_{2}^{c}\cap I_{1}\subseteq I_{1}$, $S_{1}^{c}\cap I_{3}\subseteq I_{3}$, $ I_{1}\cap I_{3}\subseteq I_{1}$ and $I_1\cap I_3\subseteq I_{3}$, it follows that  $I= (S_{1}^{c}\cap I_{3})\cup (S_{2}^{c}\cap I_{1})\cup (I_{1}\cap I_{3})\in \mathcal{I}_1\vee \mathcal{I}_2$.  It follows that $A=[(S_1^c\cap S_2^c)\cup I]^c\cup (I_2\cup I_4)=(S_1\cup S_2)\setminus I)\cup (I_2\cup I_4)$. Since $\mathcal{S}$ is a semigroup of sets then $S_1\cup S_2\in \mathcal{S}$, and thus $A\in \mathcal{S}* (\mathcal{I}_{1}\vee \mathcal{I}_{2})$.

\item  Since $\mathcal{I}_{1} \subseteq \mathcal{I}_{1}\vee \mathcal{I}_{2}$ and  $\mathcal{I}_{2} \subseteq \mathcal{I}_{1}\vee \mathcal{I}_{2}$ the inclusions $ \mathcal{I}_{1}* \mathcal{S}\subseteq (\mathcal{I}_{1}\vee \mathcal{I}_{2})* \mathcal{S}$ and $ \mathcal{I}_{2}* \mathcal{S}\subseteq (\mathcal{I}_{1}\vee \mathcal{I}_{2})* \mathcal{S}$ follow directly.

Assume that $A\in (\mathcal{I}_{1}\vee \mathcal{I}_{2})* \mathcal{S}$. Then $A= \left[(I_{1}\cup I_{2})\setminus S_{1}\right]\cup S_{2}$ for some $S_{1}, S_{2}\in \mathcal{S}, I_{1}\in \mathcal{I}_{1}$ and $I_{2}\in \mathcal{I}_{2}$. We can write  $A= \left[(I_{1}\cup I_{2})\cap S_{1}^{c}\right]\cup S_{2} = (I_{1}\cap S_{1}^{c})\cup (I_{2}\cap S_{1}^{c})\cup S_{2}= (I_{1}\setminus S_{1})\cup(I_{2}\setminus S_{1}) \cup S_{2}$. Observe that  $I_{1}\setminus S_{1} \in \mathcal{I}_{1}$ and  $I_{2}\setminus S_{2}
\in \mathcal{I}_{2}$. Put $K= I_{1}\setminus S_{1}$ and $L= I_{2}\setminus S_{2}$. This implies that  $A= (K\cup L)\cup S_{2} =(K\cup S_2)\cup (L\cup S_2)=[(K\setminus S_2)\cup S_2]\cup [(L\setminus S_2)\cup S_2]$ and hence $A\in 
 (\mathcal{I}_{1}* \mathcal{S}) \vee (\mathcal{I}_{2}* \mathcal{S})$.

Assume that $A\in  (\mathcal{I}_{1}* \mathcal{S}) \vee (\mathcal{I}_{2}* \mathcal{S})$. Then  $A= \left[( I_{1}\setminus S_{1})\cup S_{2}\right]\cup \left[ (I_{2}\setminus S_{3})\cup S_{4}\right]$, where $I_{1}\in \mathcal{I}_{1}, S_{1},S_{2}, S_{3},S_{4} \in \mathcal{S}$ and $I_{2}\in \mathcal{I}_{2}$. Note that $A= \left[(I_{1}\setminus S_{1})\cup (I_{2}\setminus S_{3})\right]\cup (S_{2}\cup S_{4})$. Since  $I_{1}\setminus S_{1}\in \mathcal{I}_{1}$ and $I_{2}\setminus S_{3}\in \mathcal{I}_{2}$, it follows that the set  $I=(I_{1}\setminus S_{1})\cup (I_{2}\setminus S_{3})$ is an element of $\mathcal{I}_1\vee \mathcal{I}_2$. Since $\mathcal{S}$ is a semigroup of sets then the set  $S=S_2\cup S_4 $ is an element of $\mathcal{S}$. It follows that $A= I\cup S= (I\setminus S)\cup S$ and hence $A\in  (\mathcal{I}_{1}\vee \mathcal{I}_{2})* \mathcal{S}$.
\end{enumerate}
\end{proof}

\begin{corollary}\label{Lebesgue215}
Let $\mathcal{I}_1, \mathcal{I}_2,\cdots, \mathcal{I}_n$ be a finite collection of ideals of sets on $X$. If $\mathcal{S}$ is a semigroup of sets on $X$, then the following inclusion and equality always hold: 
\begin{enumerate}[(i).]
	\item $\bigvee_{i=1}^{n}(\mathcal{S}*\mathcal{I}_{i}) \subseteq \mathcal{S}* (\bigvee_{i=1}^{n}\mathcal{I}_{i})$.
		\item $(\bigvee_{i=1}^{n}\mathcal{I}_{i})*\mathcal{S}=\bigvee_{i=1}^{n}(\mathcal{I}_{i}*\mathcal{S})$.
\end{enumerate}
\end{corollary}

\subsection{Lebesgue measurability and the Baire property} 
Recall that the \emph{Lebesgue outer measure} of a set $E \subseteq \mathbb{R}$, denoted by $\mu^{*}(E)$, is meant the number  $\mu^{*}(E)=\inf\left\{\sum_{n=1}^{\infty} \ell(I_n):  E\subseteq \bigcup_{n=1}^{\infty}I_n\right\}$,  where $\inf $ is taken over all sequences $\{I_n\}_{n=1}^{\infty}$ consisting of open intervals covering the set $E$. The Lebesgue outer measure is defined for all subsets of $\mathbb{R}$, but it is not countably additive.  

A subset $E$ of $\mathbb{R}$ is said to be \emph{Lebesgue measurable}, if for each $A\subseteq \mathbb{R}$, the equality $\mu^{*}(A)=\mu^{*}(A\cap E)+\mu^{*}(A\cap E^c)$ is satisfied. If $E$ is a Lebesgue measurable set, then the Lebesgue measure of $E$ is its outer measure, and it is denoted by $\mu(E)$. Let $\mathcal{N}_0$ be the collection of all subsets of $\mathbb{R}$ having the Lebesgue measure zero (null subsets of $\mathbb{R}$).  It is well known that the family
 $\mathcal{N}_0$ is a $\sigma$-ideal of sets on $\mathbb{R}$. The family $\mathcal{L}(\mathbb{R})$ of all Lebesgue measurable sets on $\mathbb{R}$ is a $\sigma$-algebra of sets on $\mathbb{R}$, containing the collection $\mathcal{B}_O(\mathbb{R})$ of all Borel sets on $\mathbb{R}$ as well as  the collection $\mathcal{N}_0$.  Let us note that there exist subsets of $\mathbb{R}$ which are not measurable in the Lebesgue sense \cite{Vi}, \cite{Kh1}. Hence the domain $\dom(\mu)$ of the set function $\mu$ is not equal to  $\mathcal{P}(\mathbb{R})$. As a consequence, the complement $\mathcal{L}^c(\mathbb{R})=\mathcal{P}(\mathbb{R})\setminus \mathcal{L}(\mathbb{R})$ of $\mathcal{L}(\mathbb{R})$ in $\mathcal{P}(\mathbb{R})$ is not empty. 
 
Recall \cite{Mo} that, if $\mathcal{O}(\mathbb{R})\subseteq \mathcal{P}(\mathbb{R})$ and $\Psi(\mathbb{R})$ is a group of homeomorphisms of $\mathbb{R}$ onto itself, then the family $\mathcal{O}(\mathbb{R})$ is said to be  \emph{invariant} under the action of $\Psi(\mathbb{R})$, if for each $A\in \mathcal{O}(\mathbb{R})$ and for each $h\in \Psi(\mathbb{R})$, we have $h(A)\in \mathcal{O}(\mathbb{R})$. 

The collection $\mathcal{L}(\mathbb{R})$ is invariant under the action of the group $\Phi(\mathbb{R})$ of all translations of $\mathbb{R}$; i.e., if $E\in \mathcal{L}(\mathbb{R})$ and $t\in \mathbb{R}$, then $E+t:=\{e+t: e\in E\}\in \mathcal{L}(\mathbb{R})$. Furthermore, $\mu(E+t)=\mu(E)$, and if $E\in \mathcal{L}(\mathbb{R}), 
t\in \mathbb{R}$, then $tE:=\{te: e\in E\}\in \mathcal{L}(\mathbb{R})$ and $\mu(tE)=|t|\mu(E)$, where $|t|$ is the absolute value of $t$. Hence, the families $\mathcal{L}(\mathbb{R})$ and $\mathcal{L}^c(\mathbb{R})$ are invariant under the action of the group $\Pi(\mathbb{R})$ for which elements are of the form $h(x)=ax+b$ with $a,b\in \mathbb{R}$ and $a\neq 0$.

\begin{lemma}[\cite{ME}]\label{Lebesgue216}
Let $A$ and $B$ be subsets of $\mathbb{R}$. If $A\in \mathcal{L}(\mathbb{R})$ and $\mu(A\Delta B)=0$, then $B\in \mathcal{L}(\mathbb{R})$ and $\mu(A)=\mu(B)$.
\end{lemma}
If $A$ a subset of $\mathbb{R}$, then $\inte (A)$ and $\cl(A)$ are used to denote the interior and the closure of $A$ in $\mathbb{R}$, respectively. A subset $M$ of $\mathbb{R}$ is said to be \emph{meager} (or \emph{of first category}), if it can be represented as a countable union of nowhere dense; i.e., $M=\bigcup_{i=1}^{\infty} M_i$ with $\inte \cl(M_i)=\emptyset$ for each $i=1,2, \cdots$.  It is well known that the collection $\mathcal{I}_m$ of all first category subsets of $\mathbb{R}$ is a $\sigma$-ideal of sets on $\mathbb{R}$ and that $\mathcal{I}_f \subsetneq \mathcal{I}_c\subsetneq \mathcal{I}_m$ and $\mathcal{I}_f \subsetneq \mathcal{I}_c\subsetneq \mathcal{N}_0$. 

A subset $A$ of $\mathbb{R}$ is said to have the \emph{Baire property} in $\mathbb{R}$ if $A$ can be represented as  $A=O\Delta M$, where $O$ is open in $\mathbb{R}$, $M$ is a first category set on $\mathbb{R}$, and $\Delta$ is the usual operation of standard symmetric difference of sets \cite{Ox}, \cite{KK}. The family $\mathcal{B}_P(\mathbb{R})$ of sets having the Baire property in $\mathbb{R}$ is a $\sigma$-algebra, containing the collections $\mathcal{B}_O(\mathbb{R})$ and $\mathcal{I}_m$, and it is invariant under the action of the group $\mathcal{H}(\mathbb{R})$ of all homeomorphisms of $\mathbb{R}$ onto itself. The Lebesgue measurability and the Baire property are two important classical concepts in real analysis and topology.

\subsection{Vitali selectors in the additive topological group of real numbers}\label{Eliyezeri}
The Vitali selectors of $\mathbb{R}$ constitute an example of subsets of $\mathbb{R}$, which are not Lebesgue measurable and without the Baire property in $\mathbb{R}$. To define Vitali selectors,  we follow \cite{Kh1} and \cite{Vi}, and we emphasize that their existence is granted by the Axiom of Choice.

Consider a countable dense subgroup $Q$ of the additive topological group $(\mathbb{R},+)$ of real numbers. Define a relation  $\mathfrak{R}$ on $\mathbb{R}$ as follows:  for $x,y\in \mathbb{R}$, let $x \mathfrak{R}y$ if and only if $x-y\in {Q}$. Clearly, $\mathfrak{R}$ is an equivalence relation on $\mathbb{R}$, and hence it divides $\mathbb{R}$ into equivalence classes. Let $\mathbb{R}/{Q}=\{E_{\alpha}(Q): \alpha \in I\}$ be the collection of all equivalence classes, where $I$ is some indexing set.  Accordingly, we have the following decomposition of $\mathbb{R}$:
\begin{equation}\label{Fourier1}
\mathbb{R}=\bigcup \{ E_\alpha (Q): \alpha \in I\}.
\end{equation}
 It follows from the definition of $\mathfrak{R}$ that the set $\mathbb{R}/{Q}$ consists of disjoint translated copies of $Q$ by elements of $\mathbb{R}$. Namely, if $t\in E_\alpha (Q)$ and $E_\alpha(Q) \in \mathbb{R}/Q$, then $E_\alpha(Q)= Q+t=\{q+t: q\in Q\}$. Hence each equivalence class $E_\alpha (Q)$ is a countable dense subset of  $\mathbb{R}$.  Equality~\eqref{Fourier1}  implies that $\card (I)= \mathfrak{c}$, where $\mathfrak{c}$ is the continuum.
\begin{example}\label{Lebesgue217}
The set $\mathbb{Q}$ of rational numbers, the set $\mathbb{D}=\{a+b\sqrt{2}: a, b\in 2\mathbb{N}\}$, the set $\mathbb{Q}(\gamma)=\{a+b \gamma: a,b \in \mathbb{Q}\}$ for each irrational number $\gamma$, and the set $\sqrt{2}\mathbb{Q}=\{\sqrt{2}q: q\in \mathbb{Q}\}$, are some examples of countable dense subgroups of $(\mathbb{R},+)$. 
\end{example}

\begin{definition}[\cite{Vi}, \cite{Kh1}]\label{Lebesgue218}
A \emph{{Vitali selector}} of $\mathbb{R}$ related to $Q$ is any subset $V$ of $\mathbb{R}$ containing one element for each equivalence class; i.e. any subset $V$ of $\mathbb{R}$ for which $\card (V\cap E_\alpha(Q))=1$ for each $\alpha \in I$.  A Vitali selector is called a 
\emph{Vitali set}, whenever the subgroup ${Q}$ coincides with  the additive group $\mathbb{Q}$ of rational numbers.
\end{definition}

\begin{proposition}[\cite{Vi}, \cite{Kh1}]\label{Lebesgue219}
Let $Q$ be a countable dense subgroup of $(\mathbb{R},+)$ and let $V$ be a Vitali selector related to $Q$.  Then, the following statements hold:
\begin{enumerate}[(i)]
\item If $q_1,q_2\in Q$ and $q_1\neq q_2$, then $(V+q_1)\cap (V+q_2)=\emptyset$.
\item Any two sets in the collection $\{V+q: q\in Q\}$ are homeomorphic, and 
\begin{equation}\label{Fourier2}
\mathbb{R}=\bigcup \{V+q: q\in Q\}.
\end{equation}
\item The set $V$ is not of the first category in $\mathbb{R}$, and it is not a null set.
\end{enumerate}
\end{proposition}
The following theorem shows that the collection of all Vitali selectors of $\mathbb{R}$ is invariant under the action of the group $\Phi (\mathbb{R})$ of all translations of $\mathbb{R}$.

\begin{theorem}[\cite{CN2}, \cite{VH}]\label{Lebesgue220}
If $U=\bigcup_{i=1}^n V_i$, where each $V_i$ is a Vitali selector related to $Q_i$ and $t\in \mathbb{R}$, then the set $U+t:=\left(\bigcup_{i=1}^n V_i\right)+t=\bigcup_{i=1}^n (V_i+t)$ is a union of Vitali selectors, where each $V_i+t$ is related to $Q_i$, for $i=1,2, \cdots,n$.
\end{theorem}

It is possible to define bounded and unbounded Vitali selectors of $\mathbb{R}$. If $O$ is a non-empty open set of $\mathbb{R}$, then one can define Vitali selectors which are dense in $O$. This implies, in particular, that there exist Vitali selectors which are dense in $\mathbb{R}$. We further point out that there exists a Vitali selector which contains a perfect set \cite{Mo}, and it can be easily observed that for any Vitali selector $V$ of $\mathbb{R}$ the set $\mathbb{R}\setminus V$ is dense in $\mathbb{R}$.

\begin{theorem}[\cite{Kh1}, \cite{Vi}]\label{Lebesgue221}
Any Vitali selector $V$ of $\mathbb{R}$ is not measurable in the Lebesgue sense and does not have the Baire property in $\mathbb{R}$.
\end{theorem}
It follows from Theorem \ref{Lebesgue221} and Proposition \ref{Lebesgue219} that if $V$ be a Vitali selector of $\mathbb{R}$, then every Lebesgue measurable subset of $V$ has the Lebesgue measure zero, and every subset of $V$ with the Baire property is of the first category.

The next theorem is a more general result on the Baire property than Theorem \ref{Lebesgue221}. 

\begin{theorem}[\cite{Ch}]\label{Lebesgue222}
Let $V_i$ be a Vitali selector for each $i\leq n$, where $n$ is some integer such that $n\geq 1$. Then the set $U=\bigcup_{i=1}^n V_i$ does not contain the difference $O\setminus M$, where $O$ is a non-empty open set and $M$ is a meager. In particular, the set $U$ does not possess the Baire property in $\mathbb{R}$.
\end{theorem}
A similar result to Theorem \ref{Lebesgue222} about non-Lebesgue measurability of finite unions of Vitali selectors of $\mathbb{R}$  was proved by A.B. Kharazishvili.

\begin{theorem}[\cite{Kh2}]\label{Lebesgue223}
If $\{V_\alpha: 1\leq \alpha \leq m\}$ is a non-empty finite family of Vitali selectors of $\mathbb{R}$, then the union $\bigcup \{V_\alpha: 1\leq \alpha \leq m\}$ is not measurable in the Lebesgue sense.
\end{theorem}

Below, we recall the classical Banach Theorem, and two important  lemmas, which were used in the proof of Theorem \ref{Lebesgue223}, and they will be very useful in the sequel. For, first recall that $\mathcal{B}_b (\mathbb{R})$ is denoting the family of all bounded subsets of $\mathbb{R}$.
 
\begin{theorem}[Banach Theorem \cite{Kh2}]\label{Lebesgue224}
Let $\mathcal{R}$ be a translation invariant ring of subsets of $\mathbb{R}$, satisfying the relations $\mathcal{R} \subseteq \mathcal{B}_b (\mathbb{R})$ and $[0, 1) \in \mathcal{R}$, and let $\vartheta: \mathcal{R}\longrightarrow [0,+\infty)$ be a finitely additive translation invariant functional such that $\vartheta ([0,1))=1$. Then there exists a finitely additive translation invariant functional $\eta: \mathcal{B}_b (\mathbb{R})\longrightarrow [0, +\infty)$ such that $\eta$ is an extension of $\vartheta$.
\end{theorem}

\begin{lemma}[\cite{Kh2}]\label{Lebesgue225}
Let $\vartheta$ be as in Theorem \ref{Lebesgue224}, and let $X\in \mathcal{B}_b(\mathbb{R})$ have the following property: There exists a bounded infinite sequence $\{h_k: k\in \mathbb{N}\}$ of elements of $\mathbb{R}$ such that the family $\{X+h_k: k\in \mathbb{N}\}$ is disjoint. If $X\in \dom (\vartheta)$, then necessarily $\vartheta (X)=0$.
\end{lemma}

\begin{lemma}[\cite{Kh2}]\label{Lebesgue226}
Let $X$ be a bounded subset of a Vitali selector $V$. Then, $X$ has the property indicated in Lemma \ref{Lebesgue225}.
\end{lemma}

It follows from Equality \ref{Fourier2} in Proposition~\ref{Lebesgue219}, that the results of Theorem \ref{Lebesgue222}  and Theorem \ref{Lebesgue223},  are not valid for infinite countable unions of Vitali selectors of $\mathbb{R}$.  However,  the following theorem provides examples of infinite countable unions of Vitali selectors without the Baire property in $\mathbb{R}$.

\begin{theorem}[\cite{AN}]\label{Lebesgue227}
If $V$ is a Vitali selector of $\mathbb{R}$  related to $Q$ and $\Gamma$ is a non-empty proper subset of ${Q}$, then the set $U=\bigcup \{V+q: q\in \Gamma \}$ does not possess the Baire property in $\mathbb{R}$.
\end{theorem}

\begin{question}\label{Lebesgue228}
\emph{Let $V$ be a Vitali selector of $\mathbb{R}$ related to $Q$ and let $\Gamma$ be an infinite countable proper subset of $Q$. Under what conditions the set $U=\bigcup \{V+q: q\in \Gamma \}$ is not measurable in the Lebesgue sense?}
\end{question}

It is clear that if $Q\setminus \Gamma$ is finite, then by Theorem \ref{Lebesgue223}, the set $W=\bigcup \{V+q: q\in Q\setminus \Gamma \}$ is not measurable in the Lebesgue sense. Consequently, by Equality \ref{Fourier2}, the set $U=\bigcup \{V+q: q\in \Gamma\}$ is also not measurable in the Lebesgue sense. 

We also note \cite{CN1} that, if  $O$ is a non-empty open subset of $\mathbb{R}$, then there exists a sequence $\{V_1, V_2, \cdots \}$ of (disjoint) Vitali  selectors of $\mathbb{R}$ such that 
$O=\bigcup_{i=1}^{\infty} V_i$. It is clear that such a union is Lebesgue measurable and has the Baire property in $\mathbb{R}$.

\subsection{Bernstein sets of the additive topological group of real numbers}
The Bernstein sets on $\mathbb{R}$ constitute also an example of subsets of $\mathbb{R}$ which are not measurable in the Lebesgue sense. A Bernstein set is a subset of $\mathbb{R}$ that meets every uncountable closed subset of $\mathbb{R}$ but that contains none of them.

\begin{definition}[\cite{Kh1}, \cite{Ox}]
A subset $B$ of $\mathbb{R}$ is called a \emph{Bernstein set} if $B\cap F\neq \emptyset$ and $(\mathbb{R}\setminus B)\cap F\neq \emptyset$ for each uncountable closed subset $F$ of $\mathbb{R}$.  
\end{definition}
The existence and the construction of Bernstein sets on $\mathbb{R}$ is based on the Method of Transfinite Recursion as it is detailed in \cite{Ox}. For Bernstein sets, the following statements hold.

\begin{proposition}[\cite{Ox}, \cite{Ts}]\label{Lebesgue229}
Let $B$ be a Bernstein subset of $\mathbb{R}$. Then the following statements hold.
\begin{enumerate}[(i)]
\item The complement $\mathbb{R}\setminus B$ of $B$ is also a Bernstein set, and 
$\inte (B)=\inte (\mathbb{R}\setminus B)=\emptyset$.
\item Both $B$ and $\mathbb{R}\setminus B$ are dense subsets of  $\mathbb{R}$ and  we have $\card (B)=\card (\mathbb{R}\setminus B)=\mathfrak{c}$.
\end{enumerate}
\end{proposition}
The family  $\mathcal{B}_E(\mathbb{R})$ of all Bernstein subsets of $\mathbb{R}$ is invariant under the action of the group $\mathcal{H}(\mathbb{R})$ of all homeomorphisms of $\mathbb{R}$ onto itself, i.e. if $B\in \mathcal{B}_E(\mathbb{R})$ and $h\in \mathcal{H}(\mathbb{R})$ then $h(B)\in \mathcal{B}_E(\mathbb{R})$.

\begin{theorem}[\cite{Ox}]\label{Lebesgue231}
Any Bernstein set $B$ on $\mathbb{R}$ is not measurable in the Lebesgue sense and does not have the Baire property. Indeed, every Lebesgue measurable subset of either $B$ or $\mathbb{R}\setminus B$ has the Lebesgue measure zero, and every subset of either $B$ or $\mathbb{R}\setminus B$ with the Baire property is of the first category.
\end{theorem}

\begin{corollary}[\cite{Ox},\cite{Ga}]\label{Lebesgue232}
 Any set with positive Lebesgue measure has a non-Lebesgue measurable subset.  Any set of second category has a subset that does not have the Baire property.
\end{corollary}

We point out that there exist Bernstein subsets of $\mathbb{R}$, which have some additional algebraic structures for subgroups of the additive group $(\mathbb{R},+)$, as it is indicated in the following statements.

\begin{lemma}[\cite{Kh1}, \cite{MC}]\label{Lebesgue233}
There exists a subgroup $B$ of $(\mathbb{R},+)$, such that the factor group $\mathbb{R}/B$ is isomorphic to the group $(\mathbb{R},+)$, and $B$ is a Bernstein set in $\mathbb{R}$.
\end{lemma}

\begin{theorem}[\cite{Kh1}]\label{Lebesgue234}
There exist two subgroups $G_1$ and $G_2$ of the additive group $(\mathbb{R},+)$ such that $G_1\cap G_2=\{0\}$, and both $G_1$ and $G_2$ are Bernstein sets in $\mathbb{R}$. 
\end{theorem}
For other notions and facts, we refer the reader to \cite{CN4} and  \cite{Ox}.

\section{Semigroups of non-Lebesgue measurable sets generated by Bernstein sets}
\subsection{Semigroups related to a  Bernstein subgroup of $(\mathbb{R},+)$}
Let $B$ be a Bernstein subset of $\mathbb{R}$ which has the algebraic structure of being a subgroup of $(\mathbb{R},+)$ as in Lemma \ref{Lebesgue233}. Consider the collection $\mathbb{R}/B=\{B+x: x\in \mathbb{R}\}$ of all cosets of $B$. Without losing generality, we may assume that $\mathbb{R}/B$ consists of pairwise disjoint sets, and for simplicity the collection $\mathbb{R}/B$ will be denoted by $\mathcal{B}$. Hence, $\mathcal{B}$ is made of pairwise translated copies of $B$ by real numbers. From \cite{Kh3}, we observe that 
$\card (\mathcal{B}) \geq \aleph_0$,  where $\aleph_0=\card(\mathbb{N})$, and $\card (\mathcal{B})$ is the same as the cardinality of the set $\{\mathbb{R}\setminus Y: Y\in \mathcal{B}\}$. Since the family $\mathcal{B}_E(\mathbb{R})$ is invariant under the action of the group $\mathcal{H}(\mathbb{R})$, it follows that each element of $\mathcal{B}$ is also a Bernstein set on $\mathbb{R}$.  

Let $\mathcal{S}(\mathcal{B})=\{\bigcup_{i=1}^n B_i: B_i\in \mathcal{B}, n\in \mathbb{N}\}$ be the semigroup of sets generated by  $\mathcal{B}$. Evidently, the family $\mathcal{S}(\mathcal{B})$ is invariant under the action of the group $\Phi(\mathbb{R})$ of all translations of $\mathbb{R}$ onto itself.

\begin{lemma}\label{Lebesgue31}
If $U$ is an element of the semigroup $\mathcal{S}(\mathcal{B})$, then $U$ is a Bernstein subset of $\mathbb{R}$. Consequently, the semigroup $\mathcal{S}(\mathcal{B})$ consists of sets which are not measurable in the Lebesgue sense, and without the Baire property in $\mathbb{R}$.
\end{lemma}

\begin{proof}
Assume that $U\in \mathcal{S}(\mathcal{B})$. Then, there exists $n\in \mathbb{N}$ such that $U=\bigcup_{i=1}^n B_i$, where $B_i\in \mathcal{B}$ for $i=1,2,\cdots,n$.
Let $F$ be an uncountable closed subset of $\mathbb{R}$. 

 Since each $B_i$ is a Bernstein subset of $\mathbb{R}$, then we have $F\cap B_i\neq \emptyset$ for each $i=1,2,\cdots, n$. It follows that $F\cap U=F\cap \left( \bigcup_{i=1}^n B_i\right)=\bigcup_{i=1}^n (F\cap B_i)\neq \emptyset$.

Assume that $F\cap (\mathbb{R}\setminus U)=\emptyset$. Then $F\subseteq U$.  Let $B_k$ be an element of $\mathcal{B}$ for some $k\notin \{1,2, \cdots, n\}$. Such an element exists, since $\card(\mathcal{B})\geq \aleph_0$. Since each element of $\mathcal{B}$ is a Bernstein set, we must have $F\cap B_k\neq \emptyset$ but by construction, $B_k\cap U=B_k\cap (\bigcup_{i=1}^n B_i)=\emptyset$ implying that $\emptyset \neq F\cap B_k\subseteq U\cap B_k=\emptyset$ . Hence the inclusion $F\subseteq U$ is impossible. Necessarily, we must have $F\cap (\mathbb{R}\setminus U)\neq \emptyset$.

We conclude that $U$ is a Bernstein subset of $\mathbb{R}$. As a Bernstein set, $U$ is not measurable in the Lebesgue sense and does not  have the Baire property in $\mathbb{R}$.
\end{proof}

\begin{proposition}\label{Lebesgue33}
The families $\mathcal{S}(\mathcal{B})*\mathcal{N}_0$ and $\mathcal{N}_0*\mathcal{S}(\mathcal{B})$ are semigroups of sets on $\mathbb{R}$ such that $\mathcal{S}(\mathcal{B})\subseteq \mathcal{N}_0*\mathcal{S}(\mathcal{B})\subseteq \mathcal{S}(\mathcal{B})*\mathcal{N}_0$. They are invariant under the action of the group $\Phi(\mathbb{R})$, and they consist of sets which are not measurable in the Lebesgue sense.
\end{proposition}

\begin{proof}
The families are semigroups of sets by Proposition \ref{Lebesgue25} and the inclusions follow from the same proposition. 

Let $A\in \mathcal{S}(\mathcal{B})*\mathcal{N}_0$ and assume that $A\in \mathcal{L}(\mathbb{R})$. Then $A=(U\setminus M)\cup N$, where $U\in \mathcal{S}(\mathcal{B})$ and $M,N\in \mathcal{N}_0$.  Note that $A\setminus U\subseteq N$ and $U\setminus A\subseteq M$ and hence $A\Delta U \subseteq M\cup N$. It follows that $\mu\left[(A\setminus U)\cup (U\setminus A)\right]=\mu(A\Delta U)\leq \mu (M\cup N)=0$ and thus $\mu(A\Delta U)=0$. Lemma \ref{Lebesgue216} indicates that the set $U$ must be measurable in the Lebesgue sense. However, $U$ is a Bernstein set on $\mathbb{R}$, and thus, it is not measurable in the Lebesgue sense. This is a contradiction.

The family $\mathcal{S}(\mathcal{B})*\mathcal{N}_0$ is invariant under the action of the group $\Phi(\mathbb{R})$, since both families $\mathcal{S}(\mathcal{B})$ and $\mathcal{N}_0$ are invariant under the action of the group $\Phi(\mathbb{R})$.
\end{proof}

\begin{lemma}\label{Lebesgue35}
Let $Y$ be a bounded subset of a Bernstein set $A$ in the collection $\mathcal{B}$. Then, $Y$ has the property indicated in Lemma \ref{Lebesgue225}.
\end{lemma}

\begin{proof}
Consider a Bernstein set $A\in \mathcal{B}$, and let $Y$ be a bounded subset of $A$. Then, we have $A=B+x_0$, for some $x_0\in \mathbb{R}$, where $B$ is a Bernstein set having an algebraic structure of being a subgroup of $(\mathbb{R}, +)$ as described in Lemma \ref{Lebesgue233}, and $Y+x\subseteq B+x_0+x=B+y$ for $y=x_0+x\in \mathbb{R}$.  In view of the definition of $\mathcal{B}$, the family $\mathcal{B}=\mathbb{R}/B=\{B+x: x\in \mathbb{R}\}$ of all cosets is made by pairwise disjoint sets. Consequently, the family $\{Y+x: x\in \mathbb{R}\}$ consists of pairwise disjoint sets.  Since every infinite set contains an infinitely countable subset \cite{SL}, let $\Lambda$ be an infinitely countable bounded subset of $\mathbb{R}$. The family $\{x_k: x_k \in \Lambda, k=1,2,\cdots\}$ can play the role of $\{h_k: k\in \mathbb{N}\}$ in Lemma \ref{Lebesgue225}. It follows that if $Y\in \dom (\vartheta)$ then $\vartheta (Y)=0$, and this ends the proof.
\end{proof}

\begin{proposition}\label{Lebesgue36}
Let $B$ be a Bernstein set of $\mathbb{R}$ that has the algebraic structure of being a subgroup of $(\mathbb{R},+)$. Any element $U$ of the semigroup $\mathcal{S}(\mathcal{B})$, cannot contain any set of strictly positive Lebesgue measure.
\end{proposition}

\begin{proof}
Suppose that there exists a Lebesgue measurable subset $Y$ of $\mathbb{R}$ such that $\mu(Y)>0$ and $Y\subseteq U$.  Since $U\in \mathcal{S}(\mathcal{B})$, then $U=\bigcup_{i=1}^n B_i$ with $B_i\in \mathcal{B}$ for each $i=1,2,\cdots, n$.  Since $Y=\bigcup_{r=-\infty}^{\infty}\left(Y\cap [r,r+1)\right)$ and $\mu(Y)>0$, then we have
$\mu \left( Y\cap [r,r+1) \right)>0$ for some $r$. Without loss of generality, we may assume that $Y$ is bounded.  Let $\vartheta$ be the restriction of $\mu$ to the family 
$\mathcal{B}_b(\mathbb{R})\cap \dom(\mu)$. Note that the family $\mathcal{B}_b(\mathbb{R})\cap \dom(\mu)$ is a ring of sets on $\mathbb{R}$. So, there exits a functional $\eta$ as in Theorem \ref{Lebesgue224} extending $\vartheta$ on $\mathcal{B}_b(\mathbb{R})$. Then

\begin{equation}\label{Fourier3}
\begin{split}
0<\vartheta (Y)=\eta (Y)=\eta (Y\cap U)=\eta \left[Y\cap \left(\bigcup_{i=1}^n B_i\right)\right]\\=\eta \left[\bigcup_{i=1}^n \left(Y\cap B_i\right)\right]\leq \sum_{i=1}^n \eta (Y\cap B_i)
\end{split}
\end{equation}

Inequality \ref{Fourier3} implies that $\eta (Y\cap B_i)>0$ for some $i\in \{1,2,\cdots,n\}$. Since $Y\cap B_i$ is a bounded subset of the Bernstein set $B_i\in \mathcal{B}\subseteq \mathcal{S}(\mathcal{B})$, then it has the property described in Lemma \ref{Lebesgue35}. According to Lemma \ref{Lebesgue35}, we must have the equality $\eta (Y\cap B_i)=0$, and this is a contradiction. 
\end{proof}

\begin{corollary}\label{Lebesgue37}
Let $B$ be a Bernstein subset of $\mathbb{R}$ which has an algebraic structure of being a subgroup of $(\mathbb{R},+)$. Any  element of the family $\mathcal{S}(\mathcal{B})*\mathcal{N}_0$, cannot contain any set of strictly positive Lebesgue measure.
\end{corollary}

\begin{proof}
Consider $A\in \mathcal{S}(\mathcal{B})*\mathcal{N}_0$. Note that $A=(U\setminus M)\cup N\subseteq U\cup N$, where $U\in \mathcal{S}(\mathcal{B})$ and $M,N\in \mathcal{N}_0$. Assume that there exists $Y\subseteq A$ such that $\mu(Y)>0$. Note that $Y=(Y\cap U)\cup (Y\cap N)$. This implies that $0<\mu(Y)\leq \mu(Y\cap U)+\mu(Y\cap N)=\mu(Y\cap U)$. Hence  the set $Y\cap U$ is a subset of $U$ with a strictly positive Lebesgue measure, which is impossible by Proposition \ref{Lebesgue36}.
\end{proof}
Observe that a contradiction in the proof of Corollary \ref{Lebesgue37} can be obtained by considering Lemma \ref{Lebesgue31} and Theorem \ref{Lebesgue231}.

\begin{theorem}\label{Lebesgue38}
Let $U_k$ be an element of $\mathcal{S}(\mathcal{B})$ and $h_k$ be an element of $\Phi(\mathbb{R})$ for $k=1,2,\cdots,n$, where $n\in \mathbb{N}$. Then, the set $U=\bigcup_{k=1}^n h_k(U_k)$ is not measurable in the Lebesgue sense, and it does not possess the Baire property in $\mathbb{R}$.
\end{theorem}

\begin{proof}
It is enough to show that the set $U=\bigcup_{k=1}^nh_k(U_k)$ is a Bernstein set on $\mathbb{R}$.  Accordingly, we will show that $U\in\mathcal{S}(\mathcal{B})$.  

Since $U_k\in \mathcal{S}(\mathcal{B})$, then $U_k=\bigcup_{i=1}^m B_{ki}$, where $B_{ki}\in \mathcal{B}$ for $i=1,2,\cdots,m$ and some  $m\in \mathbb{N}$. We can write the following:$$U=\bigcup_{k=1}^n h_k(U_k)=\bigcup_{k=1}^n h_k\left(\bigcup_{i=1}^m B_{ki}\right)=\bigcup_{k=1}^n\left[\bigcup_{i=1}^m h_k(B_{ki})\right]$$ By the invariance of the family $\mathcal{B}$ under the action $\Phi(\mathbb{R})$, each set $h_k(B_{ki})$ is also an element of $\mathcal{B}$, and hence, a Bernstein set. Put $B_k= \bigcup_{i=1}^m h_k(B_{ki})$. Since the family $\mathcal{S}(\mathcal{B})$ is invariant under the action of $\Phi(\mathbb{R})$, it follows that $B_k$ is an element of $\mathcal{S}(\mathcal{B})$.  Now, the set $U=\bigcup_{k=1}^n B_k$ is a finite union of elements of $\mathcal{S}(\mathcal{B})$. Since $\mathcal{S}(\mathcal{B})$ is a semigroup of sets, then we have  $U\in \mathcal{S}(\mathcal{B})$. 

Lemma \ref{Lebesgue31} implies that $U$ is a Bernstein set  and consequently, it is not measurable in the Lebesgue sense, and it does not possess the Baire property in $\mathbb{R}$.
\end{proof}

\subsection{Semigroups related to two  Bernstein subgroups of $(\mathbb{R},+)$}
Let $B_1$ and $B_2$ be Bernstein sets having an algebraic structure of being subgroups of $(\mathbb{R},+)$ as in Theorem \ref{Lebesgue234}, and consider the families $\mathcal{B}_1=\mathbb{R}/B_1$ and  and $\mathcal{B}_2=\mathbb{R}/B_2$ of all disjoint translates (cosets) of $B_1$ and $B_2$, respectively. Let $\mathcal{S}(\mathcal{B}_1)$ and $\mathcal{S}(\mathcal{B}_2)$ be the semigroups of sets generated by $\mathcal{B}_1$ and $\mathcal{B}_2$, respectively. 

\begin{lemma}\label{Lebesgue310}
Let $B_1$ and $B_2$ be Bernstein sets having the algebraic structure of being subgroups of $(\mathbb{R},+)$, and let  $U_1\in \mathcal{S}(\mathcal{B}_1)$ and $U_2\in \mathcal{S}(\mathcal{B}_2)$. Then, the union $U=U_1\cup  U_2$ cannot contain any subset of strictly positive Lebesgue measure.
\end{lemma}

\begin{proof}
Assume that there exists a Lebesgue measurable set $Y$ such that $\mu(Y)>0$ and $Y\subseteq U=U_1\cup U_2$, where $U_1=\bigcup_{i=1}^n B_{1i}$ and $U_2=\bigcup_{k=1}^m B_{2k}$ with $B_{1i}\in \mathcal{B}_1$ and $B_{2k}\in \mathcal{B}_2$ for $i=1,2,\cdots,n$ and $k=1,2,\cdots,m$. Without loss of generality, we may assume that the set $Y$ is bounded. It follows from Proposition \ref{Lebesgue36}, that the set $Y$ cannot be placed entirely in $U_1$ nor in $U_2$. Write $U=\bigcup_{i=1}^{n+m} X_i$, where $X_i=B_{1i}$ for $i=1,2, \cdots, n$ and $X_{n+k}=B_{2k}$ for $k=1,2,\cdots,m$. Accordingly, we have the following:

\begin{equation}\label{Fourier4}
0<\mu(Y)=\mu(Y\cap U)=\mu\left[\bigcup_{i=1}^{n+m} (Y\cap X_i)\right]\leq \sum_{i=1}^{n+m}\mu(Y\cap X_i).
\end{equation}
Inequality \ref{Fourier4} implies that $\mu(Y\cap X_i) >0$ for some index $i \in \{1,2, \cdots, n+m\}$. Since $Y\cap X_i \in \mathcal{B}_b(\mathbb{R})$, let $\vartheta$ be the restriction of $\mu$ to the ring $\mathcal{B}_b(\mathbb{R})\cap \dom (\mu)$. For this $\vartheta$, there exists a functional $\eta$ as in Theorem \ref{Lebesgue224} which is an extension of $\vartheta$. It follows that $0<\mu(Y\cap X_i)=\vartheta (Y\cap X_i)=\eta (Y\cap X_i)$.  We will have two cases to consider:
\begin{enumerate}[$\bullet$]
\item If $X_i\in \mathcal{B}_1$, then we must have $\eta (Y\cap X_i)=0$ by Lemma \ref{Lebesgue35}. 
\item If $X_i\in \mathcal{B}_2$, then we must have $\eta (Y\cap X_i)=0$ by Lemma \ref{Lebesgue35}. 
\end{enumerate}
We conclude that the set $Y$ cannot exist and this ends the proof.
\end{proof}

\begin{theorem}\label{Lebesgue311}
The family  $\mathcal{S}(\mathcal{B}_1) \vee  \mathcal{S}(\mathcal{B}_2)$ is a semigroup of sets on $\mathbb{R}$, which consists of non-Lebesgue measurable sets, and it is invariant under the action of the group  $\Phi(\mathbb{R})$.
\end{theorem}

\begin{proof}
The family $\mathcal{S}(\mathcal{B}_1) \vee  \mathcal{S}(\mathcal{B}_2)$ is a semigroup by Lemma \ref{Lebesgue26}. Let $U\in \mathcal{S}(\mathcal{B}_1) \vee  \mathcal{S}(\mathcal{B}_2)$ and assume that $U$ is a Lebesgue measurable set. Then, $U=U_1\cup U_2$, where $U_1\in \mathcal{S}(\mathcal{B}_1)$ and $U_2\in \mathcal{S}(\mathcal{B}_2)$. 
\begin{enumerate}[$\bullet$]
\item Assume that $\mu(U)=0$. Then $\mu(U_1)=\mu(U_2)=0$. This is a contradiction of Theorem \ref{Lebesgue231}, since $U_1$ and $U_2$ are Bernstein sets by Lemma \ref{Lebesgue31}.
\item The inequality $\mu(U)>0$ is impossible by Lemma \ref{Lebesgue310}.
\end{enumerate}
It follows that  the set $U$ is not Lebesgue measurable. 

It is evident that  for any $h\in \Phi(\mathbb{R})$ we have $h(U)=h(U_1\cup U_2)=h(U_1)\cup h(U_2)\in \mathcal{S}(\mathcal{B}_1) \vee  \mathcal{S}(\mathcal{B}_2)$, due to the fact that both  families $\mathcal{S}(\mathcal{B}_1)$ and  $\mathcal{S}(\mathcal{B}_2)$ are invariant under the action of the group $\Phi (\mathbb{R})$.
\end{proof}

\begin{corollary}\label{Lebesgue312}
If $ A\in \mathcal{S}(\mathcal{B}_1) \vee  \mathcal{S}(\mathcal{B}_2)$, then $\dim A=0$, where $\dim $ is the Lebesgue covering dimension.
\end{corollary}

\begin{proof}
Assume that there is an element $A$ in $\mathcal{S}(\mathcal{B}_1) \vee  \mathcal{S}(\mathcal{B}_2)$ such that $\dim A=1$. Note that $A=U_1\cup U_2$, where $U_1\in \mathcal{S}(\mathcal{B}_1)$ and $U_2\in \mathcal{S}(\mathcal{B}_2)$. Recall that each subset of $\mathbb{R}$ with $\dim A=1$ contains a non-empty open subset of $\mathbb{R}$  \cite{ENG}. It follows that there exists a non-empty open set $O$ in $\mathbb{R}$ such that $O\subseteq A$.  Since every non-empty open set has a positive Lebesgue measure \cite{BAR}, it follows that $\mu (O)>0$. This is a contradiction to  Lemma \ref{Lebesgue310}, and consequently, since $A\neq \emptyset$, we must have $\dim A=0$.
\end{proof}

\begin{theorem}\label{Lebesgue313}
The families  $\mathcal{N}_0*\left(\mathcal{S}(\mathcal{B}_1) \vee  \mathcal{S}(\mathcal{B}_2)\right) $ and
$\left(\mathcal{S}(\mathcal{B}_1) \vee  \mathcal{S}(\mathcal{B}_2)\right) *\mathcal{N}_0$ are semigroups of sets on $\mathbb{R}$, and they satisfy the inclusions: $\mathcal{S}(\mathcal{B}_1) \vee  \mathcal{S}(\mathcal{B}_2) \subseteq \mathcal{N}_0*\left(\mathcal{S}(\mathcal{B}_1) \vee  \mathcal{S}(\mathcal{B}_2)\right) \subseteq \left(\mathcal{S}(\mathcal{B}_1) \vee  \mathcal{S}(\mathcal{B}_2)\right)*\mathcal{N}_0$. They are invariant under the action of the group $\Phi(\mathbb{R})$, and they consist of sets which are not measurable in the Lebesgue sense.
\end{theorem}

\begin{proof}
The families are semigroups by Proposition \ref{Lebesgue25}. The inclusions follow from the same proposition. The family $\left(\mathcal{S}(\mathcal{B}_1) \vee  \mathcal{S}(\mathcal{B}_2)\right) *\mathcal{N}_0$ is invariant under the action of the group $\Phi(\mathbb{R})$   by Theorem \ref{Lebesgue311} and the fact that the collection $\mathcal{N}_0$ is invariant under the action of the group $\Phi(\mathbb{R})$. The proof that each element of the family $\left(\mathcal{S}(\mathcal{B}_1) \vee  \mathcal{S}(\mathcal{B}_2)\right)*\mathcal{N}_0$ is not measurable in the Lebesgue sense goes in the same line as in Proposition \ref{Lebesgue33} taking into consideration Theorem \ref{Lebesgue311}.
\end{proof}

The family $\mathcal{S}(\mathcal{B}_1)*\mathcal{S}(\mathcal{B}_1)$ does not need to be  a semigroup of sets, but the following statement shows that, it consists of elements, which are not measurable in the Lebesgue sense.

\begin{corollary}\label{Lebesgue314}
Each element of the family $\mathcal{S}(\mathcal{B}_1) * \mathcal{S}(\mathcal{B}_2)$ is not measurable in the Lebesgue sense.
\end{corollary}

\begin{proof}
Let $A\in \mathcal{S}(\mathcal{B}_1) * \mathcal{S}(\mathcal{B}_2)$ and assume that $A$ is a Lebesgue measurable set. Note that  $A=(U_1\setminus U_2)\cup U_3$ for some $U_1 \in \mathcal{S}(\mathcal{B}_1)$ and $U_2, U_3\in \mathcal{S}(\mathcal{B}_2)$. Since $U_3$ is  Bernstein set and $U_3\subseteq A$, then the  set $A$ cannot have the Lebesgue measure zero. Assume that $\mu (A)>0$. It follows that  $A=(U_1\setminus U_2)\cup U_3\subseteq U_1\cup U_3\in \mathcal{S}(\mathcal{B}_1)\vee \mathcal{S}(\mathcal{B}_2)$. By Lemma \ref{Lebesgue310}, the set $U_1\cup U_2$ can not contain any set of strictly positive Lebesgue measure.
\end{proof}

\begin{question}\label{Lebesgue316}
\emph{Is each element $U$ in the families  $\mathcal{S}(\mathcal{B}_1)\vee \mathcal{S}(\mathcal{B}_2)$  and  $\mathcal{S}(\mathcal{B}_1)* \mathcal{S}(\mathcal{B}_2)$ without the Baire property in $\mathbb{R}$?}
\end{question} 

\section{Semigroups of non-Lebesgue measurable sets generated by Vitali selectors}
Let $\mathcal{C}$ be the family of all countable dense subgroups of $(\mathbb{R}, +)$.  The following statement shows that, each finite union of Vitali selectors is not measurable in the Lebesgue sense, and it generalizes Theorem \ref{Lebesgue223}.

\begin{theorem}[\cite{VH}]\label{Lebesgue41}
Let $U=\bigcup_{i=1}^n V_i$ be a finite union of Vitali selectors of $\mathbb{R}$, where $V_i\in \mathcal{V}(Q_i)$ and each $Q_i$ is an element of $\mathcal{C}$ for $i=1,2,\cdots,n$. Then, the set $U$ is not measurable in the Lebesgue sense.
\end{theorem}
For $n=2$, Theorem \ref{Lebesgue41} implies the following statement.

\begin{corollary}\label{Lebesgue42}
Suppose that $V_1$ and $V_2$ are Vitali selectors related to elements $Q_1$ and $Q_2$ respectively in $\mathcal{C}$. Then, at least one of the sets $V_1\setminus V_2, V_2\setminus V_1$ and $V_1\cap V_2$ must be a non measurable set in the Lebesgue sense.
\end{corollary}

\begin{proof}
It follows from Theorem \ref{Lebesgue41} that the set  $V_1\cup V_2$ is not measurable in the Lebesgue sense. Note that $V_1\cup V_2=(V_1\setminus V_2)\cup (V_2\setminus V_1)\cup (V_1\cap V_2)$ and sets in this union are disjoint. If all the sets in this union are Lebesgue measurable, then $V_1\cup V_2$ will be a Lebesgue measurable set, and this will be a contradiction.
\end{proof}
A result similar to Theorem \ref{Lebesgue41} also holds in the case of the Baire property, as it is proved in \cite{CN4}. If $Q\in \mathcal{C}$, then  we denote by $\mathcal{V}(Q)$ the family of all Vitali selectors related to $Q$, and $\mathcal{V}_1(Q)$ the semigroup generated by  $\mathcal{V}(Q)$; that is $\mathcal{V}_1(Q)=\{\bigcup_{i=1}^n V_i: V_i\in \mathcal{V}(Q), n\in \mathbb{N}\}$, i.e. the collection of all finite unions of elements of $\mathcal{V}(Q)$. The following statement shows that each topological group isomorphism maps Vitali selectors of $\mathbb{R}$ to Vitali selectors of $\mathbb{R}$, not necessarily related to the same subgroups of $(\mathbb{R}, +)$.

\begin{theorem}[\cite{VH}]\label{Lebesgue43}
Let $Q$ be a countable dense subgroup $(\mathbb{R},+)$ and let $V\in \mathcal{V}(Q)$. If $h: (\mathbb{R}, +) \longrightarrow (\mathbb{R}, +)$ is a topological group isomorphism, then $P=h(Q)\in \mathcal{C}$ and $W=h(V) \in \mathcal{V}(P)$.
\end{theorem}

Let $Q_1$ and $Q_2$ be elements of $\mathcal{C}$ such that $Q_1\subseteq Q_2$ and $Q_1\neq Q_2$. It was shown in \cite{CN1} that if $\card (Q_2/Q_1)<\infty$ then $\mathcal{V}_1(Q_1)\subseteq \mathcal{V}_1(Q_2)$ and $\mathcal{V}_1(Q_1)\neq \mathcal{V}_1(Q_2)$, and if $\card (Q_2/Q_1)=\aleph_0$ then $\mathcal{V}_1(Q_1)\cap\mathcal{V}_1(Q_2)=\emptyset$. From here, we consider the collection  $\mathcal{V}=\{V: V\in \mathcal{V}(Q), Q\in \mathcal{C}\}$ of  all Vitali selectors of $\mathbb{R}$, and we define the semigroup $\mathcal{S}(\mathcal{V})=\{\bigcup_{i=1}^n V_i: V_i\in \mathcal{V}, n\in \mathbb{N} \}$ generated by the collection $\mathcal{V}$ of all Vitali selectors of $(\mathbb{R}, +)$.  Clearly, $\mathcal{V}(Q)\subsetneq \mathcal{V}$ and $\mathcal{V}_1(Q)\subsetneq \mathcal{S}(\mathcal{V})$ for each $Q\in \mathcal{C}$.  It is well known \cite{CN4} that the families $\mathcal{S}(\mathcal{V})$, $\mathcal{I}_m*\mathcal{S}(\mathcal{V})$  and $\mathcal{S}(\mathcal{V})*\mathcal{I}_m$ consist of sets without the Baire property and they are invariant under the action of $\Phi(\mathbb{R})$.

 \begin{theorem}[\cite{VH}]\label{Lebesgue44}
The families $\mathcal{N}_0*\mathcal{S}(\mathcal{V})$ and $\mathcal{S}(\mathcal{V})*\mathcal{N}_0$ are semigroups of sets on $\mathbb{R}$, for which elements are not measurable in the Lebesgue sense, such that $\mathcal{S}(\mathcal{V})\subsetneq \mathcal{N}_0*\mathcal{S}(\mathcal{V})\subsetneq \mathcal{S}(\mathcal{V})*\mathcal{N}_0$, and they are  invariant under the action of the group $\Pi(\mathbb{R})$ of all affine transformations of $\mathbb{R}$ onto itself.
\end{theorem}

The following theorem is a more general result than Theorem \ref{Lebesgue41}.
\begin{theorem}\label{Lebesgue45}
Let $U=\bigcup_{i=1}^n V_i$ be a finite union of Vitali selectors of $\mathbb{R}$, where $V_i\in \mathcal{V}(Q_i)$ and each $Q_i$ is an element of $\mathcal{C}$ for $i=1,2,\cdots,n$. Then, the set $U$ cannot contain any subset of strictly positive Lebesgue measure.
\end{theorem}

\begin{proof}
Suppose that there exists a Lebesgue measurable subset $Y$ of $\mathbb{R}$ such that $\mu(Y)>0$ and $Y\subseteq U$. Without loss of generality, we may assume that the set $Y$ is bounded. Let $\vartheta$ be the restriction of $\mu$ to the ring of sets $\mathcal{B}_b(\mathbb{R})\cap \dom(\mu)$. For this $\vartheta$, there exists a functional $\eta$ as in Theorem \ref{Lebesgue224}. 
Clearly, we have 
\begin{equation}\label{Fourier5}
\begin{split}
0<\vartheta (Y)=\eta (Y)=\eta (Y\cap U)=\eta \left[Y\cap \left(\bigcup_{i=1}^n V_i\right)\right]\\ =\eta \left[\bigcup_{i=1}^n \left(Y\cap V_i\right)\right]\leq \sum_{i=1}^n \eta (Y\cap V_i)
\end{split}
\end{equation}

Inequality \ref{Fourier5} implies that $\eta (Y\cap V_i)>0$ for some $i\in \{1,2,\cdots,n\}$.
Since $Y\cap V_i$ is a bounded subset of the Vitali selector $V_i$, it follows from Lemma \ref{Lebesgue226}, that it has the property described in Lemma \ref{Lebesgue225}. According to Lemma \ref{Lebesgue225}, we must have the equality $\eta (Y\cap V_i)=0$, but this a contradiction. We conclude that the set $U$ cannot contain any Lebesgue measurable set with positive measure.
\end{proof}

It is important to point out that Theorem \ref{Lebesgue45} is not valid in the case of countable unions of Vitali selectors. A simple way to observe this fact, is to consider Equality \ref{Fourier2}, but a more general result was proved in \cite{CN1}, where it was shown  that,  if $O$ is a non-empty open subset of $\mathbb{R}$, then there exist a sequence of Vitali selectors $V_1, V_2, \cdots$ such that $O=\bigcup_{i=1}^{\infty} V_i$. Such a union contains a set of strictly positive Lebesgue measure.

\begin{corollary}\label{Lebesgue46}
No element of the family $\mathcal{S}(\mathcal{V})*\mathcal{N}_0$ can contain any set of strictly positive Lebesgue measure.
\end{corollary}

With the help of Theorem \ref{Lebesgue45}, we can prove the following statement.

\begin{corollary}\label{Lebesgue47}
Let $U_k=\bigcup_{i=1}^n V_{i k}$ be a finite union of Vitali selectors of $\mathbb{R}$, where $V_{ik}\in \mathcal{V}(Q_k)$ and $Q_k$ is an element of $\mathcal{C}$ for each $k$,  and let $h_k$ be a topological group isomorphism of $(\mathbb{R},+)$ onto itself, for each $k=1,2, \cdots,n$. Then, the union $U=\bigcup_{k=1}^m h_k(U_k)$, cannot contain any Lebesgue measurable set of strictly positive measure.
\end{corollary}

\begin{proof}
By Theorem \ref{Lebesgue45}, it is enough to show that $U$ is a finite union of Vitali selectors of $\mathbb{R}$. Note that  $h_k(U_k)=h_k\left(\bigcup_{i=1}^n V_{ik}\right)=\bigcup_{i=1}^n h_k\left(V_{ik}\right)$. By Theorem \ref{Lebesgue43}, the set $h_k(V_{ik})$ is a Vitali selector related to the group $h_k(Q_k)\in \mathcal{C}$. Hence, $h_k(U_k)$ is a finite union of Vitali selectors. It follows that $U=\bigcup_{k=1}^m h_k(U_k)= \bigcup_{k=1}^m \bigcup_{i=1}^n h_k(V_{ik})$ is also a  finite union of Vitali selectors of $\mathbb{R}$.
\end{proof}

 \section{Semigroups of non-Lebesgue measurable sets  generated by a combination of Bernstein sets and Vitali selectors}
 We now combine Bernstein sets and Vitali selectors of $\mathbb{R}$, to construct families of sets for which elements are not measurable in the Lebesgue sense.
 
\begin{theorem}\label{Lebesgue51}
Let $B$ be a Bernstein subset of $\mathbb{R}$ which has the algebraic structure of being a subgroup of $(\mathbb{R},+)$, and let $\mathcal{S}(\mathcal{V})$ be the semigroup generated by the collection $\mathcal{V}$ all Vitali selectors of $\mathbb{R}$. Any union $U=U_1\cup U_2$, where $U_1\in \mathcal{S}(\mathcal{B})$ and $U_2\in \mathcal{S}(\mathcal{V})$, cannot contain any subset of strictly positive Lebesgue measure.
\end{theorem}

\begin{proof}
Assume that there exists a Lebesgue measurable set $Y$ such that $\mu(Y)>0$ and $Y\subseteq U=U_1\cup U_2$, where $U_1=\bigcup_{i=1}^n B_i$ and $U_2=\bigcup_{k=1}^m V_k$ for $B_i\in \mathcal{B}$ and $V_k\in \mathcal{V}$. Without loss of generality, we may assume that the set $Y$ is bounded. It follows from Proposition \ref{Lebesgue36}  and Theorem \ref{Lebesgue45} that the set $Y$ cannot stay entirely in $U_1$ nor in $U_2$. Write $U=\bigcup_{i=1}^{n+m} X_i$, where $X_i=B_i $ for $i=1,2, \cdots, n$ and $X_{n+k}=V_k$ for $k=1,2,\cdots, m$. Then we have the following:
\begin{equation}\label{Fourier6}
0<\mu(Y)=\mu(Y\cap U)=\mu\left[\bigcup_{i=1}^{n+m} (Y\cap X_i)\right]\leq \sum_{i=1}^{n+m}\mu(Y\cap X_i).
\end{equation}
It follows from Inequality \ref{Fourier6} that $\mu(Y\cap X_i) >0$ for some index $i \in\{1,2, \cdots, n+m\}$. Since $Y\cap X_i\in \mathcal{B}_b(\mathbb{R})$, let $\vartheta$ be the restriction of $\mu$ to the ring of sets  $\mathcal{B}_b(\mathbb{R})\cap \dom (\mu)$. For this $\vartheta$ there exists a functional $\eta$ as in Theorem \ref{Lebesgue224} which is an extension of $\vartheta$. So we have $0<\mu(Y\cap X_i)=\vartheta(Y\cap X_i)=\eta (Y\cap X_i)$. 
\begin{enumerate}[$\bullet$]
\item If $X_i$ is an element of $\mathcal{V}$, then $\eta (Y\cap X_i)=0$ by Lemma \ref{Lebesgue226}. 
\item If $X_i$ is an element of $\mathcal{B}$, then $\eta (Y\cap X_i)=0$ by Lemma \ref{Lebesgue35}.
\end{enumerate}
As a conclusion the set $Y$ cannot exist.
\end{proof}

\begin{corollary}\label{Lebesgue52}
Let $B$ be a Bernstein set of $\mathbb{R}$ which has the algebraic  structure of being a subgroup of $(\mathbb{R},+)$, and let $\mathcal{S}(\mathcal{V})$ be the semigroup generated by the collection $\mathcal{V}$ all Vitali selectors of $\mathbb{R}$. Then the semigroup $\mathcal{S}(\mathcal{B})\vee \mathcal{S}(\mathcal{V})$ consists of sets which are not measurable in the Lebesgue sense, and it is invariant under the action of the group $\Phi(\mathbb{R})$.
\end{corollary}

\begin{proof}
Assume that there exists a Lebesgue measurable set $U$ in $\mathcal{S}(\mathcal{B})\vee \mathcal{S}(\mathcal{V})$. Then, $U=U_1\cup U_2$, where $U_1\in \mathcal{S}(\mathcal{B})$ and $U_2\in \mathcal{S}(\mathcal{V})$.  Since $U_2=\bigcup_{i=1}^n V_i$, where $V_i\in \mathcal{V}(Q_i)$, let $V_k$ be a fixed Vitali selector in this union such that  $V_k\in \mathcal{V}(Q_k)$.  Since by Equality \ref{Fourier2}, $\mathbb{R}=\bigcup \{V_k+q: q\in Q_k\}$ and $V_k\subseteq U_2\subseteq U$, then we have $\mathbb{R}=\bigcup \{U+q: q\in Q_k\}$.
 Given that $\mu(U+q)=\mu(U)$ for each $q\in Q_k$ and $\mu(\mathbb{R})>0$, then we must have $\mu(U)>0$, and this contradicts Theorem \ref{Lebesgue51}.

Since the family $\mathcal{S}(\mathcal{V})$ is invariant under the action of the group $\Pi(\mathbb{R})$ and the family  $\mathcal{S}(\mathcal{B})$  invariant under the action of the group $\Phi(\mathbb{R})$, it follows that the family $\mathcal{S}(\mathcal{B})\vee \mathcal{S}(\mathcal{V})$ is invariant under the action of the group $\Phi(\mathbb{R})$.
\end{proof}

\begin{corollary}\label{Lebesgue53}
Each element of the family $\mathcal{S}(\mathcal{B}) * \mathcal{S}(\mathcal{V})$ is not measurable in the Lebesgue sense.
\end{corollary}

\begin{proof}
Let $A\in \mathcal{S}(\mathcal{B}) * \mathcal{S}(\mathcal{V})$ and assume that $A$ is a Lebesgue measurable set. Then, $A=(U_1\setminus U_2)\cup U_3$ for some $U_1 \in \mathcal{S}(\mathcal{B})$ and $U_2, U_3\in \mathcal{S}(\mathcal{V})$. Since $U_3$ is a finite union of Vitali selectors and $U_3\subseteq A$, then the set $A$ cannot have the Lebesgue measure zero. Assume that $\mu (A)>0$. It follows that  $A=(U_1\setminus U_2)\cup U_3\subseteq U_1\cup U_3\in \mathcal{S}(\mathcal{B}_1)\vee \mathcal{S}(\mathcal{V})$. But by Theorem \ref{Lebesgue51}, the set $U_1\cup U_2$ cannot contain any set of strictly positive Lebesgue measure.
\end{proof}
Let us point out that the family $\mathcal{S}(\mathcal{B}) * \mathcal{S}(\mathcal{V})$  does not need to be a semigroup of sets.

\begin{theorem}\label{Lebesgue56}
The families $ \mathcal{N}_0*\left(\mathcal{S}(\mathcal{V})\vee \mathcal{S}(\mathcal{B})\right)$ and $\left(\mathcal{S}(\mathcal{V})\vee \mathcal{S}(\mathcal{B})\right)*\mathcal{N}_0$ are semigroups of sets on $\mathbb{R}$ satisfying the inclusions:  $\mathcal{S}(\mathcal{V})\vee \mathcal{S}(\mathcal{B}) \subseteq \mathcal{N}_0*\left(\mathcal{S}(\mathcal{V})\vee \mathcal{S}(\mathcal{B})\right) \subseteq \left(\mathcal{S}(\mathcal{V})\vee \mathcal{S}(\mathcal{B})\right)*\mathcal{N}_0$. They are invariant under the action of the group $\Phi(\mathbb{R})$, and they consist of sets which are not measurable in the Lebesgue sense.
\end{theorem}

\begin{proof}
The given families are semigroups of sets by Proposition \ref{Lebesgue25} and Lemma \ref{Lebesgue26}. The inclusions follow from Proposition \ref{Lebesgue25}.

Let $A\in \left(\mathcal{S}(\mathcal{V})\vee \mathcal{S}(\mathcal{B})\right)*\mathcal{N}_0$ and assume that $A$ is measurable in the Lebesgue sense. Then $A=((U_1\cup U_2)\setminus M)\cup N$, where $U_1\in \mathcal{S}(\mathcal{V}), U_2\in \mathcal{S}(\mathcal{B})$ and $M,N\in \mathcal{N}_0$. Note that $A\setminus (U_1\cup U_2)\subseteq N$ and $(U_1\cup U_2)\setminus A\subseteq M$ and hence $A\Delta (U_1\cup U_2) \subseteq M\cup N$. It follows that $\mu(A\Delta (U_1\cup U_2))\leq \mu (M\cup N)=0$ and thus $\mu(A\Delta (U_1\cup U_2))=0$. It follows from Lemma \ref{Lebesgue216} that the set $U_1\cup U_2$ must be measurable in the Lebesgue sense.  But, Corollary \ref{Lebesgue52} tells us that the set $U_1\cup U_1$ is not measurable in the Lebesgue sense, and we have a contradiction.
\end{proof}

Let us note that $\mathcal{S}(\mathcal{B})\vee \mathcal{S}(\mathcal{V})\subseteq \mathcal{S}(\mathcal{B}) * \mathcal{S}(\mathcal{V})$ and $\left(\mathcal{S}(\mathcal{B})\vee \mathcal{S}(\mathcal{V})\right)*\mathcal{N}_0\subseteq \left(\mathcal{S}(\mathcal{B}) * \mathcal{S}(\mathcal{V})\right)*\mathcal{N}_0$. It follows from Lemma \ref{Lebesgue28} and Proposition \ref{Lebesgue211} that $\mathcal{S}(\mathcal{B}\vee \mathcal{V})*\mathcal{N}_0=\left(\mathcal{S}(\mathcal{V})\vee \mathcal{S}(\mathcal{B})\right)*\mathcal{N}_0=(\mathcal{S}(\mathcal{V})*\mathcal{N}_0)\vee (\mathcal{S}(\mathcal{B})*\mathcal{N}_0)$. 
These semigroups consist of sets which are not measurable in the Lebesgue sense.

\begin{corollary}\label{Lebesgue54}
The families  $\mathcal{N}_0*\left(\mathcal{S}(\mathcal{B}) * \mathcal{S}(\mathcal{V}) \right)$ and $\left(\mathcal{S}(\mathcal{B}) * \mathcal{S}(\mathcal{V})\right)*\mathcal{N}_0$ consist of elements which are not measurable in the Lebesgue sense.
\end{corollary}

\begin{question}\label{Lebesgue55}
\emph{Is each element of the families $\mathcal{S}(\mathcal{B}) \vee \mathcal{S}(\mathcal{V})$ and $\mathcal{S}(\mathcal{B}) * \mathcal{S}(\mathcal{V})$ without the Baire property in $\mathbb{R}$?}
\end{question}

The positive answer to Question \ref{Lebesgue55} will imply that the semigroups of sets  $\mathcal{S}(\mathcal{V})\vee \mathcal{S}(\mathcal{B}), \mathcal{I}_m*\left(\mathcal{S}(\mathcal{V})\vee \mathcal{S}(\mathcal{B})\right)$ and $\left(\mathcal{S}(\mathcal{V})\vee \mathcal{S}(\mathcal{B})\right)*\mathcal{I}_m$;
which are invariant under the action of the group $\Phi(\mathbb{R})$, consist of sets without the Baire property in $\mathbb{R}$.

\begin{lemma}\label{Lebesgue57}
Let $B_1$ and $B_2$ be Bernstein subsets of $\mathbb{R}$ having the algebraic structure of being subgroups of $(\mathbb{R},+)$, and let $\mathcal{S}(\mathcal{V})$ be the semigroup generated by all Vitali selectors of $\mathbb{R}$. Then, any union $U=U_1\cup U_2\cup U_3$, where $U_1\in \mathcal{S}(\mathcal{B}_1), U_2\in \mathcal{S}(\mathcal{B}_3)$ and $U_3\in \mathcal{S}(\mathcal{V})$ cannot contain any set of strictly positive Lebesgue measure. In particular, the family  $\mathcal{S}(\mathcal{B}_1)\vee \mathcal{S}(\mathcal{B}_2)\vee \mathcal{S}(\mathcal{V})$ is a semigroup of sets for which elements are  not measurable in the Lebesgue sense, and it is invariant under the action of the group $\Phi(\mathbb{R})$.
\end{lemma}

\begin{proof}
To prove that the element $U$ cannot contain any set of positive measure, we proceed as in Theorem \ref{Lebesgue51}. It is clear that the family $\mathcal{S}(\mathcal{B}_1)\vee \mathcal{S}(\mathcal{B}_2)\vee \mathcal{S}(\mathcal{V})$ is a semigroup of sets by Lemma \ref{Lebesgue26}. It is invariant under the action of the group $\Phi (\mathbb{R})$, since both families $\mathcal{S}(\mathcal{B}_1),\mathcal{S}(\mathcal{B}_2)$ and  $\mathcal{S}(\mathcal{V})$ are invariant under the action of the group $\Phi (\mathbb{R})$.
\end{proof}

The following statement can be proved in a similar way as Theorem \ref{Lebesgue56} by taking into account Lemma \ref{Lebesgue57}.

\begin{theorem}\label{Lebesgue58}
Let $B_1$ and $B_2$ be Bernstein subsets of $\mathbb{R}$ having the algebraic structure of being subgroups of $(\mathbb{R},+)$, and let $\mathcal{S}(\mathcal{V})$ be the semigroup generated by all Vitali selectors of $\mathbb{R}$. Then the families $\mathcal{N}_0*(\mathcal{S}(\mathcal{B}_1)\vee \mathcal{S}(\mathcal{B}_2)\vee \mathcal{S}(\mathcal{V}))$ and $(\mathcal{S}(\mathcal{B}_1)\vee \mathcal{S}(\mathcal{B}_2)\vee \mathcal{S}(\mathcal{V}))*\mathcal{N}_0$ are semigroups of sets on $\mathbb{R}$ such that 
$\mathcal{S}(\mathcal{B}_1)\vee \mathcal{S}(\mathcal{B}_2)\vee \mathcal{S}(\mathcal{V})\subseteq \mathcal{N}_0*(\mathcal{S}(\mathcal{B}_1)\vee \mathcal{S}(\mathcal{B}_2)\vee \mathcal{S}(\mathcal{V}))\subseteq (\mathcal{S}(\mathcal{B}_1)\vee \mathcal{S}(\mathcal{B}_2)\vee \mathcal{S}(\mathcal{V}))*\mathcal{N}_0$. They are  invariant under the action of the group $\Phi(\mathbb{R})$ and they consist of sets 
which are not measurable in the Lebesgue sense.
\end{theorem}

It follows from Proposition \ref{Lebesgue29} and Lemma \ref{Lebesgue28} that $\mathcal{S}(\mathcal{B}_1\vee \mathcal{B}_2\vee \mathcal{V})*\mathcal{N}_0=(\mathcal{S}(\mathcal{B}_1)\vee \mathcal{S}(\mathcal{B}_2)\vee \mathcal{S}(\mathcal{V}))*\mathcal{N}_0=(\mathcal{S}(\mathcal{B}_1) * \mathcal{N}_0)\vee (\mathcal{S}(\mathcal{B}_2)*\mathcal{N}_0)\vee (\mathcal{S}(\mathcal{V})*\mathcal{N}_0)$ and $\mathcal{N}_0*(\mathcal{S}(\mathcal{B}_1)\vee \mathcal{S}(\mathcal{B}_2)\vee \mathcal{S}(\mathcal{V}))=(\mathcal{N}_0*\mathcal{S}(\mathcal{B}_1))\vee (\mathcal{N}_0*\mathcal{S}(\mathcal{B}_2))\vee (\mathcal{N}_0*\mathcal{S}(\mathcal{V})$.  All these semigroups consist of sets which are not measurable in the Lebesgue sense. We have used the $\sigma$-ideal $\mathcal{N}_0$ in the construction of different semigroups for which elements are not Lebesgue measurable. All the statements remain valid by using an ideal of sets $\mathcal{I}$ such that $\mathcal{I}\subseteq \mathcal{N}_0$. Different statements involving the collection $\mathcal{S}(\mathcal{V})$ remain valid when $\mathcal{S}(\mathcal{V})$ is replaced by $\mathcal{V}_1(Q)$ for any countable dense subgroup $Q$ of $(\mathbb{R},+)$.

\subsection*{Acknowledgements}
We would like to thank Prof. A. B. Kharazishvili for the helpful information about Vitali selectors that he provided. We would like to express our gratitude for the financial support we got from the University of Rwanda through the UR-Sweden Program. Finally, we appreciate the referee’s insightful remarks and comments, which enabled us to present our results more effectively.

\end{document}